\documentclass[10pt]{article}
\usepackage{a4}
\usepackage[utf8]{inputenc}
\usepackage[T1]{fontenc}

\usepackage{amsmath}
\usepackage{amsthm}
\usepackage{amssymb}
\usepackage{amsopn}
\usepackage{amsfonts}
\usepackage{mathrsfs}
\usepackage{enumitem}
\usepackage{color}
\usepackage{mathtools}
\usepackage{hyperref}
\usepackage{doi}
\usepackage{graphicx}
\usepackage{bm}
\usepackage{verbatim}

\theoremstyle{plain}
\newtheorem{theorem}{Theorem}[section]
\newtheorem{proposition}[theorem]{Proposition}

\newtheorem{lemma}[theorem]{Lemma}

\theoremstyle{definition}

\theoremstyle{remark}
\newtheorem{remark}[theorem]{Remark}

\DeclareRobustCommand{\rchi}{{\mathpalette\irchi\relax}}
\newcommand{\irchi}[2]{\raisebox{\depth}{$#1\chi$}}

\DeclareMathOperator{\dist}{dist}
\DeclareMathOperator{\spt}{spt}  
\DeclareMathOperator{\into}{int}

\newcommand{\measurerestr}{%
  \raisebox{-.127ex}{\reflectbox{\rotatebox[origin=br]{-90}{$\lnot$}}}%
}
\DeclareMathOperator{\restrict}{\hspace{-0.2ex}\measurerestr\hspace{-0.2ex}}

\def\XXint#1#2#3{{\setbox0=\hbox{$#1{#2#3}{\int}$}
     \vcenter{\hbox{$#2#3$}}\kern-.5\wd0}}

\newcommand{\N}{\mathbb{N}}
\newcommand{\R}{\mathbb{R}}
\newcommand{\C}{\mathcal{C}}

\newcommand{\eps}{\varepsilon}

\begin{document}

\begin{center}
\begin{Large}
Two phase models for elastic membranes with soft inclusions 
\end{Large}
\\[0.5cm]
\begin{large}
Mario Santilli\footnote{Università degli studi dell'Aquila, Italy, {\tt mario.santilli@univaq.it}} and 
Bernd Schmidt\footnote{Universität Augsburg, Germany, {\tt bernd.schmidt@math.uni-augsburg.de}}
\end{large}
\\[0.5cm]
\today
\\[1cm]
\end{center}

\begin{abstract}
We derive an effective membrane theory in the thin film limit within a two phase material model for a specimen consisting of an elastic matrix and soft inclusions or voids. These inclusions may lead to the formation of cracks within the elastic matrix and the corresponding limiting models are described by Griffith type fracture energy functionals. We also provide simplified proofs of relaxation results for bulk materials. 
\end{abstract}

\begin{small}

\medskip

\end{small}


\section{Introduction}\label{sec:Intro}

Two phase energy functionals of the form  
\begin{equation*}
	\mathcal{E}(y, D) 
	= \int_{\Omega \setminus D} W \big( \nabla y(x) \big) \, \mathrm{d}x 
	+ \int_{\Omega \cap \partial^\ast D} \psi \big( \nu(D) \big) \, \mathrm{d} \mathcal{H}^{n-1}
\end{equation*}
naturally arise in the study of an elastic material with an unknown void or a soft inclusion. In this paper we assume that $ W $ is a  function on $ \mathbb{R}^{m \times n} $ satisfying  a standard two-sided $ L^p $ growth condition, the function $ \psi $ is a norm on $ \mathbb{R}^n $, $\Omega \subset \R^n$ is an open (Lipschitz)  domain, $y \in W^{1,p}(\Omega, \mathbb{R}^m)$ with $p > 1 $, $D $ is a set of finite perimeter contained in $ \Omega $. Recall that $ \partial^\ast D $ is the measure-theoretic boundary of $ D $ and $ \nu(D) $ is the (exterior) measure-theoretic unit-normal of $ D $. We refer to Section \ref{section:BV} for a detailed account on the notation used in this paper.

In general, the occurrence of a degenerate phase can model a variety of quite different systems of relevance. Examples include the formation of voids in a device due to mechanical or chemical degredation (or even enhancement as in Swiss cheese caused by propionic acid bacteria), soft phases  of a material such as a liquid region at the onset of a solid/liquid phase transition or a superelastic martensite phase within a shape memory alloy, cp.\ \cite{HuoMueller:93}, as well as material mixtures in which an elastic material is invaded by a (chemical or biological) substance that causes the development of extremely soft regions. Examples of the latter are the immersion of water in gypsum rock \cite{AuvrayHomandSorgi:04,Zhu_etal:19} and the resulting softening which may have severe impact on the stability of mining goafs \cite{Zhu_etal:19,Wang_etal:19} and cerebral softening \cite{KIYEA:11,NoelKuhl:19} in the light of recent elastic models of brain tissue \cite{MihaiBuddayHolzapfelKuhlGoriely:17}. At variance with other common models involving composite materials such as  \cite{FocardiGelliPonsiglione:09,BraidesSolci:13,BarchiesiLazzaroniZeppieri:16,CagnettiDalMasoScardiaZeppieri:19}, in such situations the position and geometry of the soft inclusions are not pre-assigned. 

Looking more specifically at examples of thin membranes in two phase models, which is one of the central topics of this paper, we mention that polymer electrolyte membranes (PEM) have received considerable interest recently as they constitute a basic component in PEM fuel cells. Such devices, which generate electricity from hydrogen and oxygen, promise to provide an environmentally friendly alternative to fossil fuels. Their improvement both in performance and longevity is thus a most desirable goal. As both chemical and mechanical degredation mechanisms may lead to the formation of voids (and eventually cracks) within the polymer membrane, effective models for membranes with soft inclusions are of fundamental importance in order to gain a better understanding of operational failure of such fuel cells. We refer to \cite{Singh_etal_2021} for a recent experimental investigation into membrane degredation in PEM fuel cells and the references cited therein for a broader review of the literature on membrane failure in such systems. 

We finally remark that in our theory both variable and fixed volume fractions of the two phases can be considered. This allows to describe the possibility of phase transformations as well as mixtures of different materials like alloys.

As usual in the study of energy functionals, a first basic question asks to determine the relaxation of $ \mathcal{E} $ and  this question has been first addressed in \cite{BraidesChambolleSolci:07} for quasiconvex functions $ W $. The relaxed functional is computed on (a subset of) the space of generalized special functions of bounded variations $GSBV(\Omega;\R^m)$ and it is of the form 
\[ \mathcal{E}^{\rm rel}(y, D) 
   = \int_{\Omega \setminus D} W^{\rm qc} \big( \nabla y \big) \, \mathrm{d}x 
     + 2  \int_{S_y \cap D^0} \psi(\nu(y))\, \mathrm{d}\mathcal{H}^{n-1} 
     +  \int_{\Omega \cap \partial^\ast D} \psi(\nu(D))\, \mathrm{d}\mathcal{H}^{n-1},
\] 
where $y$ has jump discontinuities along the codimension one set $S_y$ with unit-normal $\nu(y)$, $ D^0 $ is the measure-theoretic exterior of $ D $ (see Section \ref{section:BV}) and $W^{\rm qc}$ denotes the quasiconvex envelope of $W$.  We also mention that in \cite{FonsecaFuscoLeoniMillot:11} a strongly related relaxation result in two dimensions is obtained. A main motivation for such results has derived from investigations on epitaxially strained films \cite{BonnetierChambolle:02,ChambolleSolci:07}. We refer to \cite{CrismaleFriedrich:20} for latest results and a detailed account of the literature in that direction.

In the first part of this paper we give a new approach to compute this relaxed functional for every Borel  function $ W $ satisfying  a standard two-sided $ L^p $ growth condition, see \eqref{eq:p-growth}. We refer to Theorem \ref{theo:bulk-relaxation} for the precise statement. Our approach provides a considerably simpler and more direct proof of the lim inf inequality and a detailed proof of the lim sup inequality, expanding an argument which is only briefly sketched in \cite{BraidesChambolleSolci:07}. More specifically, our main contributions to the relaxation result are the following ones. 

A) In \cite{BraidesChambolleSolci:07} the proof of the lim inf inequality is based on a slicing argument. Because of the technicalities involved, this slicing argument is given in details only in the scalar case ($m=1$) and for a special choice of $ W $ and $ \psi $ (while the authors briefly indicate the necessary modifications in \cite[Remarks 6,7,8]{BraidesChambolleSolci:07} to handle the general case). On the other hand, our approach to establish the lim inf inequality is completely different, considerably simpler and it allows to directly deal with the general case without additional efforts. Instead of a slicing argument, our key idea is based on the choice of a suitable comparison functional of ``Griffith type'', which allows to obtain the sharp lower  bound from well known lower-semicontinuity results for BV elliptic functions (Theorem \ref{theo:BV-ellipticity}). 

B) We provide a detailed construction of the lim sup inequality  in the general case: we first employ some recent results on the anisotropic Minkowski content (see \cite{LussardiVilla:16} and Lemma \ref{area of distance sets}) to explicitly find the recovery sequence for regular pairs $(y,D)$ and then we pass to more general pairs with the help of suitable density results. This is a very natural argument, which has been also briefly sketched in \cite[Remark 13]{BraidesChambolleSolci:07}. On the other hand, a careful analysis of the details reveals some geometric-measure theoretic subtleties. For example, one subtle point is that it is not clear a priori that one may pass to a limit in the variables $ y $ and $ D $ simultaneously. We overcome this difficulty by passing to the limits consecutively. This allows for an application of the standard approximation result in $SBV^p_\infty$ by Cortesani-Toader, cf.\cite{CortesaniToader:99}), but then requires a deeper argument for sets of finite perimeter in terms of one-sided smooth approximations, cf.\ Theorem~\ref{theo:ComiTorres} and cp.\ \cite{ChenTorresZiemer:09,ComiTorres:17}.

We also mention that we prove that $\mathcal{E}^{\rm rel}(y, D)$ can be realized as $\lim_{k \to \infty} \mathcal{E}(y_k, D_k)$ with $y_k \to y$, $D_k \to D$ and $\mathcal{L}^n(D_k) = c_k$ for any preassigned (positive) sequence $c_k$ with $c_k \to \mathcal{L}^n(D)$. In particular, if $\mathcal{L}^n(D) = 0$ we thus obtain $\Gamma$-convergence to a pure Griffith type fracture functional of the form  
\[ y \mapsto 
\int_{\Omega} W^{\rm qc} \big( \nabla y(x) \big) \, \mathrm{d}x 
+ 2  \int_{S_y} \psi(\nu(y))\, \mathrm{d}\mathcal{H}^{n-1}, \] 
cp.\ \cite{Griffith:21,FrancfortMarigo:98,AmbrosioFuscoPallara:00}.

In the second part of this paper we further advance the theory of two phase energy functionals providing a novel analysis of thin films in the membrane limit. In particular, we focus on thin films with reference configuration $ \Omega_h = \omega \times (0, h) \subset \mathbb{R}^3 $ of small `membrane heights' $0 < h \ll 1$ and we provide a novel dimensionally reduced membrane theory in the limit $h \to 0$ for thin films consisting of an elastic matrix and a soft inclusion. Our result extends the classical work for purely elastic materials in \cite{LeDretRaoult:95} and for brittle materials in \cite{BraidesFonseca:01,BouchitteFonsecaLeoniMascarenhas:02}. This is achieved in Theorem~\ref{theo:membrane-Gamma} where we study the $ \Gamma $-convergence of appropriately renormalized versions of the functionals 
 \begin{align*}
 	\mathcal{G}_h(u,D) = \int_{\Omega_h \setminus D} W(\nabla u) \, \mathrm{d}x +  \int_{\Omega_h \cap \partial^\ast D} \psi(\nu(D))\, \mathrm{d}\mathcal{H}^{2},
 \end{align*}
where $ D \subset \Omega_h $ is a set of finite perimeter representing the shape of the voids, $ u : \Omega_h \rightarrow \mathbb{R}^3 $ is a Sobolev map representing the elastic deformation field, $ W $ is a continuous stored energy function and $ \psi $ is an arbitrary norm, which allows to model a possibly anisotropic surface energy on $\partial D \cap \Omega_h$ depending on the (exterior) normal $\nu(D)$ to $\partial^\ast D$. We explicitly compute the $ \Gamma $-limit, which is given by the functional $ \mathcal{E}_0^{\rm rel} $ on (a subset of) the space of generalized functions of bounded variations $GSBV(\omega; \mathbb{R}^3) $ of the form
 \begin{align*}
 	\mathcal{E}_0^{\rm rel}(u,D) 
 	= \int_{\omega \setminus D}W_0^{\rm qc}(\nabla u)\, \mathrm{d}x 
 	+ 2\int_{S_u \cap D^0} \psi_0(\nu(u))\, \mathrm{d}\mathcal{H}^{1}
 	+  \int_{\omega \cap \partial^\ast D} \psi_0(\nu(D))\, \mathrm{d}\mathcal{H}^{1}.
 \end{align*}
Here $W_0$ and $\psi_0$ are explicit (see equations \eqref{eq: W0} and \eqref{eq: psi0} at page \pageref{eq: psi0}) and $ W_0^{\rm qc}$ is the quasiconvex envelope of $ W_0 $. ($\psi_0$ turns out to be automatically BV elliptic.) Moreover we obtain the recovery sequence subject to volume constraints on the voids. For certain norms $ \psi $ one has that $\psi_0(\nu) = \psi(\nu,0)$ for every $ \nu \in \mathbb{R}^2 $, which leads to consider cylindrical shapes of voids in the recovery sequence. However, the general case poses some additional difficulties in the construction of recovery sequences and one finds that a crack in the limiting 2d model might typically be induced from non-cylindrical voids in the parent 3d model whose outer boundary normal has a nontrivial and non-constant out-of-plane component.

\section{Functions of bounded variation}\label{section:BV}

We collect here the notation and some basic material on generalized functions of bounded variation and sets of finite perimeter. For an exhaustive treatment of this subject we refer to \cite{AmbrosioFuscoPallara:00}. We fix a norm $ \psi $ on $ \mathbb{R}^n $ and write $ | \cdot | $ for the Euclidean norm. The dual norm of $ \psi $ is denoted 
\begin{equation*}
	\psi^\circ(u) = \max\{ \langle u,v \rangle :\psi(v) \leq 1  \} \qquad \textrm{for $ u \in \mathbb{R}^n $.}
\end{equation*}
Let $\Omega \subset \R^n$ be an open bounded set, $u : \Omega \to \R^m$ a Borel function and $x \in \Omega $. We say that $a \in \R^m$ is the {\em approximate limit} of $u$ at $ x $ if 
\[ \lim_{\rho \searrow 0} \rho^{-n} \mathcal{L}^n \big( \{ x' \in B_{\rho}(x) : |u(x') - a| > \eps \} \big) = 0 \] 
for each $\eps > 0$, in which case we write $\tilde{u}(x)$ for $a$. If this limit does not exist we say that $x$ belongs to the {\em approximate discontinuity set} $S_u$.\footnote{A caveat on notation: The set $S_u$ is called {\em weak} approximate discontinuity set and denoted $S^*_u$ in \cite{AmbrosioFuscoPallara:00}.}  For every $x \in \Omega \setminus S_u$ we say that $A \in \R^{m \times n}$ is the {\em approximate differential} of $u$ at $x \in \Omega$ if 
\[ \lim_{\rho \searrow 0} \rho^{-n} \mathcal{L}^n \big( \{ x' \in B_{\rho}(x) : |u(x') - \tilde{u}(x) - A (x' - x)| > \eps |x' - x| \} \big) = 0 \] 
for each $\eps > 0$. In this case we write $\nabla u(x)$ for $A$. 

We say that a Borel subset $ S \subset \mathbb{R}^n $ is countably $ \mathcal{H}^{n-1} $-rectifiable if there are at most countably many $C^1$ hypersurfaces of dimension $ n-1 $ in $\Omega$ that cover $S_u$ up to an $\mathcal{H}^{n-1}$ negligible set. If moreover $ \mathcal{H}^{n-1}(S) < \infty $ then we say that $ S $ is $ \mathcal{H}^{n-1} $-rectifiable. 

A function $u \in L^1(\Omega; \R^m)$ is said to lie in the space $BV(\Omega; \R^m)$ of functions of {\em bounded variation} if its distributional derivative $Du$ is a finite $\R^{m \times n}$-valued Radon measure. The total variation of $ u $ with respect to the Euclidean norm is denoted by $ |D u| $. We also need to consider \emph{the anisotropic total variation} $ \psi(Du) $ of $ D u  $ with respect to $ \psi $ for a function $ u \in BV(\Omega) $: this is the Radon measure $ \psi(D u) $ on $ \Omega $ given by 
\begin{equation*}
\psi(D u)(B) = \int_{B} \psi \bigg( \frac{D u}{|D u|}\bigg)\, \mathrm{d}|D u| \quad \textrm{for  $ B \subset \Omega $ Borel,}
\end{equation*}
where $\frac{D u}{|D u|}$ is the $ |D u| $-measurable function satisfying $ D u = \frac{D u}{|D u|} |D u| $. Setting $ \psi(D u) = + \infty $ for $ u \in L^1(\Omega) \setminus BV(\Omega) $, it follows from the Reshetnyak lower semicontinuity theorem \cite[Theorem~2.38]{AmbrosioFuscoPallara:00} that the function $ u \mapsto \psi(D u)(U) $ is lower semicontinuous in the $ L^1(U) $ topology for any open subset $U$ of $\Omega$. If $ u \in BV(\Omega; \mathbb{R}^m) $, the approximate discontinuity set $S_u$ is a countably $\mathcal{H}^{n-1}$ rectifiable set. The Lebesgue decomposition of $Du$ turns out to be $Du = \nabla u \mathcal{L}^n + D^s u$ with $\nabla u \in L^1(\Omega;\R^{m \times n})$ and singular part $D^s u$. If moreover $D^s u$ is concentrated on $S_u$, we speak of a {\em special function of bounded variation} and write $u \in  SBV(\Omega; \R^m)$. If $u \in SBV(\Omega'; \R^m)$ for all $\Omega' \subset \subset \Omega$, we write $u \in SBV_{\rm loc}(\Omega; \R^m)$. 

A function $u : \Omega \to \R^m$ is a {\em generalized function of bounded variation}, write $u \in GSBV(\Omega; \R^m)$, whenever $\varphi \circ u \in SBV_{\rm loc}(\Omega; \R^m)$ for every $\varphi \in C^1(\R^m)$ with $\spt \nabla \varphi \subset \subset \R^m$. If $m = 1$ this is equivalent to $u^M = (u \wedge M) \vee (-M) \in SBV(\Omega)$ for every $M > 0$. Moreover, for $ 1 \leq p,q \le \infty $ we define $(G)SBV^p(\Omega; \R^m)$ as the space of functions $u \in (G)SBV (\Omega;\R^m)$ for which $\nabla u \in L^p(\Omega; \R^{m \times n})$ and $\mathcal{H}^{n-1}(S_u) < \infty$ and set $(G)SBV^p_q(\Omega; \R^m) = (G)SBV^p(\Omega; \R^m) \cap L^q(\Omega; \R^m)$. In \cite[Prop.~2.3]{DalMasoFrancfortToader:05} it has been noted that $GSBV^p(\Omega;\R^m)$ is a vector space and that $u = (u_1, \ldots, u_m)$ belongs to $GSBV^p(\Omega;\R^m)$ if and only if $u_i \in GSBV^p(\Omega)$ for all $i$ and that, as a consequence, the scalar results in \cite[Sect.~4.5]{AmbrosioFuscoPallara:00} apply to show that for $u \in GSBV^p(\Omega;\R^m)$ still $S_u$ is an $\mathcal{H}^{n-1}$ rectifiable set. If we fix an approximate unit normal vector field $ \nu $ of $ S $, then $\mathcal{H}^{n-1}$ a.e.\ point $x \in S_u$ is an {\em approximate jump point} of $u$ in the sense that there are distinct $u^+(x), u^-(x) \in \R^m$ such that 
\[ \lim_{\rho \searrow 0} \rho^{-n} \mathcal{L}^n \big( \{ x' \in B_{\rho}(x) \cap H^{\pm} : |u(x') - u^{\pm}(x)| > \eps \} \big) = 0 \] 
for each $\eps > 0$, where $H^{\pm} = \{x' \in \Omega : \pm (x' - x) \cdot \nu(x) > 0 \}$. With little abuse of notation we denote each vector field $ \nu $ as above with $ \nu(u) $; notice that the triple $ (\nu(u)(x), u^+(x), u^{-}(x)) $ is uniquely determined up to a sign and a permutation. Setting $u^M = (u_1^M, \ldots, u_m^M)$ one has $S_u = \bigcup_{M > 0} S_{u^M}$ and 
\[ \nabla u_i 
   = \nabla u_i^M 
   \quad \mathcal{L}^n\text{ a.e.~on } \{|u_i| \le M \} \]
for $i = 1, \ldots, m$. The existence of one-sided traces $u^{\pm}$ is guaranteed on any countably $\mathcal{H}^{n-1}$ rectifiable set oriented by some normal field. We also note that $GSBV^p_{\infty}(\Omega;\R^m) = SBV^p_{\infty}(\Omega;\R^m)$ and that $S_u$ is the complement of the set of Lebesgue points of $u$ if $u \in L^{\infty}(\Omega;\R^m)$.

If $E \subset \R^n$ is a Borel subset then its {\em measure theoretic interior} $E^1$, {\em exterior} $E^0$ and {\em boundary} $\partial^* E$ are given by 
\begin{align*}
  E^1 
  &= \{ x \in \R^n : \lim_{\rho \searrow 0} \rho^{-n} \mathcal{L}^n \big( B_{\rho}(x) \setminus E \big) = 0 \}, \\ 
  E^0 
  &= \{ x \in \R^n : \lim_{\rho \searrow 0} \rho^{-n} \mathcal{L}^n \big( B_{\rho}(x) \cap E \big) = 0 \}, \\ 
  \partial^* E 
  &= \R^n \setminus (E^1 \cup E^0),
\end{align*} 
which are easily seen to be Borel subsets. If $ E \subset \Omega $ is a Borel set, then we say that $ E $ is a set of finite perimeter in $ \Omega $ if and only if $ \rchi_E \in BV(\Omega) $. The total variation measure of $D \rchi_E$ satisfies $|D \rchi_E| = \mathcal{H}^{n-1} \restrict (\Omega \cap \partial^* E)$. Since $S_{\rchi_E} = \partial^* E \cap \Omega$, we have $\rchi_E \in SBV(\Omega)$ and we set $ \nu(E) = \frac{D \rchi_E}{|D \rchi_E|}$. We write $\mathcal{F}(\Omega)$ to denote the collection of sets $E \subset \Omega$ of finite perimeter in $\Omega$. Moreover if $ E \in \mathcal{F}(\Omega) $ we notice that $\psi(D \rchi_E)$ is the anisotropic surface measure on $\partial^\ast E$ with density $\psi(\nu(E))$, i.e., 
\begin{equation*}
\psi(D \rchi_E) 
= \psi(\nu(E)) \mathcal{H}^{n-1} \restrict \partial^\ast E.
\end{equation*}	
We recall from \cite{Grasmair:10} the following anisotropic version of coarea formula for $ BV $ functions: if $ u \in BV(\Omega) $, then 
\begin{equation}\label{Coarea}
\psi(D u)(B) = \int_{-\infty}^{+ \infty} \psi(D \rchi_{\{u \geq t    \}}) (B)\, \mathrm{d}t
\end{equation}
for each Borel subset $ B \subset \Omega $.

We proceed to state the relevant compactness and lower semicontinuity results. The basic compactness theorem in $(G)SBV^p$ of Ambrosio is the following, cf.\ \cite{Ambrosio:90,Ambrosio:95,AmbrosioFuscoPallara:00}: 
\begin{theorem}\label{theo:Ambrosio-compactness}
Let $\Omega \subset \R^n$ be a bounded open set und $(u_k) \subset GSBV^p_q(\Omega; \R^m)$ for $p > 1$ and $q \ge 1$. Suppose that 
\[ \| u_k \|_{L^q(\Omega; \R^{m})} + \| \nabla u_k \|_{L^p(\Omega; \R^{m \times n})} + \mathcal{H}^{n-1}(S_{u_k}) \le C \]
for some constant $C > 0$. Then there exists a subsequence (not relabeled) and a $u \in GSBV^p_q(\Omega; \R^m)$ such that 
\begin{itemize}
\item[(i)] $u_k \to u$ $\mathcal{L}^n$ a.e.\ and, in case $q > 1$, in $L^1(\Omega; \R^m)$ (strongly), 
\item[(ii)] $\nabla u_k \rightharpoonup \nabla u$ in $L^p(\Omega; \R^{m \times n})$ (weakly) and 
\item[(iii)] $\liminf_{k \to 0} \mathcal{H}^{n-1}(S_{u_k}) \ge \mathcal{H}^{n-1}(S_{u})$. 
\end{itemize}
\end{theorem}

For the $\liminf$ inequalities we will make use of lower semicontinuity results. In particular, the lower semicontinuity of the bulk term follows from Kristensen's theorem in \cite{Kristensen:99}. 

\begin{theorem}\label{theo:Kristensen}
	Let $p > 1$. Suppose that $f : \R^{m \times n} \to \R$ is quasiconvex with $- C \le f(X) \le C|X|^p + C$ for all $X \in \R^{m \times n}$ and for some constant $C > 0$.  Suppose $ \Omega \subset \mathbb{R}^n $ is a bounded open set and $(u_k) \subset GSBV^p_1(\Omega;\R^m)$ is such that 
	\[ \| \nabla u_k \|_{L^p(\Omega; \R^{m \times n})} + \mathcal{H}^{n-1}(S_{u_k}) \le C \]
	for some constant $C > 0$ and $u_k \to u$ in $L^1(\Omega;\R^m)$. Then 
	\[ \liminf_{k \to \infty} \int_{\Omega} f(\nabla u_k) \, \mathrm{d}x 
	\ge \int_{\Omega} f(\nabla u) \, \mathrm{d}x. \]
\end{theorem}

	For the surface part we use the following standard result in the $SBV$ setting, see \cite{AmbrosioBraides:90b,Ambrosio:90,Ambrosio:94} and cf.\ \cite[Theorem~5.22]{AmbrosioFuscoPallara:00}. We recall that for a compact set $K \subset \R^m$ a function $g : K \times K \times \R^n \to [0,\infty)$ is {\em jointly convex}, if 
	\[ g(x, y, \nu) 
	= \sup_{h \in \N} \big[ (V_h(x) - V_h(y)) \cdot \nu \big] \] 
	for a suitable choice of $V_h \in C(K; \R^n)$, $h \in \N$. For later use we remark that in the isotropic case $g(x, y, \nu) = g(x, y) |\nu|$ this amounts to requiring 
	\[ g(x, y) 
	= \sup_{h \in \N} |V_h(x) - V_h(y)| \] 
	for suitable $V_h \in C(K)$, $h \in \N$, while for functions $g(x, y, \nu) = g(\nu)$ only depending on the crack normal this is equivalent to having $g$ even, positively 1-homogeneous and convex.

	\begin{theorem}\label{theo:BV-ellipticity}
		Let $p > 1$, $K \subset \R^m$ compact and suppose that $g : K \times K \times \R^n \to [0,\infty)$ is jointly convex with $g(x, y, \nu) \ge c|\nu|$ for a constant $c > 0$ and all $(x, y) \in K^2$ with $x \ne y$ and $\nu \in \R^n$. Suppose $ \Omega \subset \mathbb{R}^n $ is a bounded open set and $(u_k) \subset SBV^p(\Omega;\R^m)$ is such that 
		\[ \| \nabla u_k \|_{L^p(\Omega; \R^{m \times n})} \le C 
		\quad \text{and} \quad 
		u_k \in K \text{ $\mathcal{L}^n$ a.e.} \] 
		for all $k$ and some constant $C > 0$ and $u_k \to u$ in $L^1(\Omega;\R^m)$. Then 
		\[ \liminf_{k \to \infty} \int_{S_{u_k}} g \big( u_k^+, u_k^-, \nu(u_k) \big) \, \mathrm{d}\mathcal{H}^{n-1} 
		\ge \int_{S_{u}} g \big( u^+, u^-, \nu(u) \big) \, \mathrm{d}\mathcal{H}^{n-1}. \]
	\end{theorem}

For the basic density result in $SBV^p_\infty$ we define the set $\mathcal{W}(\Omega; \R^m)$ to be the space of functions $y \in SBV(\Omega;\R^m)$ such that 
\begin{itemize}
\item[(i)] $\mathcal{H}^{n-1} ( \overline{S_y} \setminus S_y ) = 0$,  
\item[(ii)] $\overline{S_y} \cap \Omega$ is the intersection of $\Omega$ with a finite union of $(n-1)$-dimensional simplices, 
\item[(iii)] $y \in W^{k,\infty}(\Omega \setminus \overline{S_y};\R^m)$ for every $k \in \N$. 
\end{itemize} 
The following density result is a special case of the main result of \cite{CortesaniToader:99}.
\begin{theorem}\label{theo:Cortesani}
Let $\Omega $ be a bounded open set with Lipschitz boundary. For every $u \in SBV^p_\infty(\Omega;\R^m) $ there exists a sequence $(u_k) \subset \mathcal{W}(\Omega; \mathbb{R}^m)$ with 
\begin{itemize}
\item[(i)] $u_k \to u$ in $L^1(\Omega;\R^m)$ and $ \limsup_{k \to \infty} \| u_k \|_\infty \leq \| u \|_\infty $, 
\item[(ii)] $\nabla u_k \to \nabla u$ in $L^p(\Omega;\R^{m \times n})$ and $ \nabla u_k(x) \to \nabla u(x) $ for $ \mathcal{L}^n $ a.e.\ $ x \in \Omega $,
\item[(iii)] $\int_{S_{u_k}} \psi(\nu(u_k))\, \mathrm{d}\mathcal{H}^{n-1} \to \int_{S_{u}} \psi(\nu(u))\, \mathrm{d}\mathcal{H}^{n-1}$, 
\item[(iii')] $\psi(\nu(u_k))\cdot \mathcal{H}^{n-1} \restrict S_{u_k}  \stackrel{*}{\rightharpoonup} \psi(\nu(u))\cdot \mathcal{H}^{n-1} \restrict S_{u}$ as Radon measures on $ \Omega $. 
\end{itemize} 
\end{theorem}

\begin{proof}
The existence of a sequence satisfying (i), (ii) and (iii) follows from \cite[Theorem~3.1 and Remark~3.2]{CortesaniToader:99}. This sequence also satisfies (iii'), as can be checked with the help of the Portmanteau theorem and combining  (iii) with the fact that, due to Theorem \ref{theo:BV-ellipticity},
\[
\liminf_{k \to \infty}\int_{S_{u_k} \cap U} \psi(\nu(u_k))\, \mathrm{d}\mathcal{H}^{n-1} \geq  \int_{S_{u}\cap U} \psi(\nu(u))\, \mathrm{d}\mathcal{H}^{n-1}
\] 
for every open set $U \subset \Omega$.
\end{proof} 

While this and related results have been widely used in Gamma convergence and relaxation results, for the sets in $\mathcal{F}(\Omega)$ will make use of a particular almost one-sided smooth approximation scheme, that has been rather recently established in \cite{ChenTorresZiemer:09,ComiTorres:17}. In the next theorem we summarize all statements on smooth approximation of sets of finite perimeter needed in the sequel. 

\begin{theorem}\label{theo:ComiTorres}
Let $E \subset \R^n$ be a bounded set of finite perimeter, $ s > 0 $ and let $ K \subset \mathbb{R}^n $ be a Borel set with $ \mathcal{H}^{n-1}(K) < \infty $. There exists a sequence of open sets $E_k$ with smooth boundaries such that for every Radon measure $\mu$ on $\R^n$ with $\mu \ll \mathcal{H}^{n-1}$ and each $r > 0$: 
\begin{itemize}
\item[(i)] $\mathcal{L}^n ( E_k \triangle E ) \to 0$, 
\item[(ii)] $ \psi(D \rchi_{E_k})(\mathbb{R}^n) \to \psi(D \rchi_E)(\mathbb{R}^n) $ and $\psi(D \rchi_{E_k}) \stackrel{*}{\rightharpoonup} \psi(D \rchi_E) $, 
\item[(iii)] $| \mu | \big( (E^1 \cup \partial^* E) \triangle \overline{E_k} \big) \to 0$, 
\item[(iv)] $\mathcal{H}^{n-1}(\partial E_k \cap K) =0 $ for every $ k $, 
 \item[(v)]  $\{x \in E : \operatorname{dist}(x, \R^n \setminus E) > r \} \subset \bigcup_k \bigcap_{m \ge k} E_m$,
\item[(vi)]  $\{x \in \R^n \setminus E : \operatorname{dist}(x, E) > r \} \subset \bigcup_k \bigcap_{m \ge k} (\R^n \setminus E_m)$,
\item[(vii)] $\{x  : \operatorname{dist}(x, \R^n \setminus E) \geq s \} \subset E_k \subset \{x :   \dist(x,E) < 2s\} $ for every $ k $.
\end{itemize} 
\end{theorem}

\begin{proof}
We fix a sequence $ \eps_k \searrow 0 $, we define $u_k = \rchi_E * \eta_{\eps_k} $ with the standard scaled mollifier $\eta_{\eps}$ and we set $ F^k_t = \{u_k > t\} $ for $ t \in (0,1) $. Since $ u_k \to \rchi_E $ in $ L^1(\Omega) $ and $ |D u_k|(\mathbb{R}^n) \to |D\rchi_{E}|(\mathbb{R}^n) $ (see \cite[p.\ 121]{AmbrosioFuscoPallara:00}), one has that $ \mathcal{L}^n(F^k_t \triangle E) \to 0  $ for all $ t \in (0,1) $ and 
\begin{equation*}
	\lim_{k \to \infty}\psi(D u_k)(\mathbb{R}^n) = \psi(D \rchi_E)(\mathbb{R}^n)  
\end{equation*}
by the Reshetnyak continuity theorem \cite[Theorem~2.39]{AmbrosioFuscoPallara:00}. Moreover it follows from the lower semicontinuity of the anisotropic total variation that for every open set $ U \subset \mathbb{R}^n $
\begin{equation*}
\liminf_{k \to \infty} \psi(D \rchi_{F^k_t})(U) 
\geq \psi(D \rchi_E)(U) \quad \textrm{for all $ t \in (0,1) $.} 
\end{equation*}
At this point we use the coarea formula in \eqref{Coarea} and Fatou's Lemma to obtain 
\begin{equation*}
\psi(D\rchi_{E})(\mathbb{R}^n) \geq \int_{0}^{1}\liminf_{k \to \infty}\psi(D\rchi_{F^k_t})(\mathbb{R}^n)\, \mathrm{d}t,
\end{equation*}
and we combine this inequality with Sard's Theorem to conclude that for $ \mathcal{L}^1 $ a.e.\ $ t \in (0,1) $ the sets $ F^k_t $ have smooth boundaries for each $ k \geq 1 $  and 
\begin{equation*}
	\liminf_{k \to \infty} \psi(D \rchi_{F^k_t})(\mathbb{R}^n) = \psi(D \rchi_E)(\mathbb{R}^n).
\end{equation*} 
Passing to a $t$-dependent subsequence that realizes the $\liminf$ as a limit and using the Portmanteau theorem we also conclude that $  \psi(D \rchi_{F^k_t})\stackrel{*}{\rightharpoonup}\psi(D \rchi_E)  $ for almost every $ t \in (0,1) $ along that sequence. Since $ \mathcal{H}^{n-1}(K) < \infty $, we infer that $ \mathcal{H}^{n-1}(\partial F^k_t \cap K) = 0 $ for every $ k \geq 1 $ and for all but countably many $ t \in (0,1) $. Moreover it is proved in \cite[Theorem~3.1]{ComiTorres:17} that $ | \mu | \big( (E^1 \cup \partial^* E) \triangle F^k_t \big) \to 0  $ for all $ t \in (0, 1/2) $. Noting that  $ F_k^t \subset \overline{F^k_t} \subset F_k^s $ for all $ 0 < s < t < \frac{1}{2} $, we conclude  $| \mu | \big( (E^1 \cup \partial^* E) \triangle \overline{F^k_t} \big) \to 0 $ for all $ t \in (0, \frac{1}{2}) $. In conclusion there exists $ t \in (0,\frac{1}{2}) $ such that for $ E_k = F^k_t $ for a suitable choice of indices $k$ depending on $t$ all the assertions in (i)-(iv) hold.   This choice also guarantees (v) and (vi). For all $ k $ sufficiently large also (vii) holds. 
\end{proof}

\section{Bulk model and relaxation}

Suppose $\Omega \subset \R^n$ is a bounded open set with Lipschitz boundary and $ 1 < p < \infty $. We associate to $y \in W^{1,p}(\Omega; \R^m)$ and any set $D \subset \Omega$ of finite perimeter an energy 
\[ \mathcal{E}(y, D)  
   =  \int_{\Omega \setminus D} W(\nabla y) \, \mathrm{d}x 
       + \int_{\Omega \cap \partial^\ast D} \psi(\nu(D))\, \mathrm{d}\mathcal{H}^{n-1}. \]
In this section we assume that $W: \mathbb{R}^{m \times n} \rightarrow \mathbb{R}$ is a Borel function which satisfies the growth condition 
\begin{align}\label{eq:p-growth}
  \bar{c} |X|^p - \bar{C} \le W(X) \le \bar{C} (1 + |X|^p) 
\end{align}
for constants $\bar{c}, \bar{C} >0$ and  $ \psi $ is an arbitrary norm on $ \mathbb{R}^n $ for which we can evidently assume $ \bar{c} |v| \leq \psi(v) \leq \bar{C} |v| $ for each $ v \in \mathbb{R}^n $.  The quasiconvex envelope $ W^{\rm qc} $ of $ W $, given by 
\begin{equation}\label{eq:Wqc-rep}
	W^{\rm qc}(X) = \inf\bigg\{   \int_{(0,1)^n}W(X + \nabla \varphi(x))\, \mathrm{d}x : \varphi \in C^\infty_c((0,1)^n, \mathbb{R}^m)  \bigg\}
\end{equation}
for every $ X \in \mathbb{R}^{m \times n} $, see \cite[Definition~6.3 and Remark~6.8]{BraidesDefranceschi:98}, satisfies the same growth condition.

Our first result identifies the relaxation of $\mathcal{E}$ with respect to $L^1$ convergence of $y$ and $\rchi_D$. It also allows for volume constraints and provides smooth recovery sequences. We set
\begin{equation*}
	\mathcal{C}(\Omega) = \{\emptyset\} \cup \big\{ A \cap \Omega: \textrm{$A \subset \mathbb{R}^n$ open with smooth boundary}, \, \mathcal{H}^{n-1}(\partial A \cap \partial \Omega) =0 \big\}.
\end{equation*}

\begin{theorem}\label{theo:bulk-relaxation}
The ($L^1$-)relaxation of $\mathcal{E}$ on $GSBV^p_1(\Omega; \R^{m}) \times \mathcal{F}(\Omega)$ is given by  
\[	\mathcal{E}^{\rm rel}(y, D) 
	=  \int_{\Omega \setminus D} W^{\rm qc} (\nabla y) \, \mathrm{d}x 
		+ 2  \int_{S_y \cap D^0} \psi(\nu(y))\, \mathrm{d}\mathcal{H}^{n-1} +  \int_{\Omega \cap \partial^\ast D} \psi(\nu(D))\, \mathrm{d}\mathcal{H}^{n-1}. 
\] 
More precisely, the following two assertions hold.
\begin{itemize} 
\item[(i)] Whenever $(y_k) \subset W^{1,p}(\Omega; \R^{m})$ and $(D_k) \subset \mathcal{F}(\Omega)$ are such that $y_k \to y$ in $L^1(\Omega; \R^{m})$ and $\rchi_{D_k} \to \rchi_{D}$ in $L^1(\Omega)$ for some $y \in GSBV^p_1(\Omega; \R^{m})$ and $D \in \mathcal{F}(\Omega)$, then one has 
\[ \liminf_{k \to \infty} \mathcal{E}(y_k, D_k) 
   \ge \mathcal{E}^{\rm rel}(y, D), \] 
 \item[(ii)]
	For each $(y, D) \in  GSBV^p_1(\Omega; \R^{m}) \times \mathcal{F}(\Omega)$ and $c_1, c_2, \ldots \in (0,\mathcal{L}^n(\Omega)]$ with $c_k \to \mathcal{L}^n(D)$ there are $(y_k) \subset C^{\infty}(\overline{\Omega}, \R^m)$ and $(D_k) \subset \mathcal{C}(\Omega)$ with $y_k \to y$ in $L^1(\Omega; \R^{m})$, $\rchi_{D_k} \to \rchi_{D}$ in $L^1(\Omega)$ and $\mathcal{L}^n(D_k) = c_k$ for all $k$ 
and
\[ \lim_{k \to \infty} \mathcal{E}(y_k, D_k) 
   = \mathcal{E}^{\rm rel}(y, D). \] 
\end{itemize}
\end{theorem}

\begin{remark}\label{Lipschitz regularity}
The Lipschitz regularity of the boundary $ \partial \Omega $ is needed in the construction of the recovery sequence. The statement in \ref{theo:bulk-relaxation}(i) holds for every bounded open set $ \Omega $.
\end{remark}

\begin{remark}
	Assume that $y \in SBV^p_\infty(\Omega; \R^m)$ in Theorem~\ref{theo:bulk-relaxation}(ii). Then we can choose the recovery sequence $(y_k) $ so that it additionally satisfies 
	\begin{equation}\label{eq: L infty bound}
	\limsup_{k\to\infty}\|y_k\|_{L^\infty}\le \|y\|_{L^\infty}.
	\end{equation}
In fact, it is easy to check that the recovery sequences constructed in Step 1 and in Step 2 of the proof of in Theorem~\ref{theo:bulk-relaxation}(ii) below satisfy the $ L^\infty $ bound in \eqref{eq: L infty bound}.
\end{remark}

\begin{remark} 
Cracks outside of $D$ can develop as a result of asymptotically thin tubular neighborhoods of $S_y \cap D^0$ whose boundary area is asymptotically twice as big as the surface area of $S_y \cap D^0$, which explains the factor $2$ occuring in the first surface term in $\mathcal{E}^{\rm rel}(y,D)$. Indeed, our proof will show that for every energetically optimal `recovery' sequence $(y_k, \rchi_{D_k}) \to (y, \rchi_D)$ in $L^1(\Omega; \R^{m}) \times L^1(\Omega)$ such that $\lim_{k \to \infty} \mathcal{E}(y_k, D_k) = \mathcal{E}^{\rm rel}(y, D)$, one has 
\[\psi(D \rchi_{D_k}) 
    \stackrel{*}{\rightharpoonup} 2  \psi(\nu(y))\cdot \mathcal{H}^{n-1}\restrict (S_y \cap D^0)\, + \psi(D \rchi_D) \]
as Radon measures on $ \Omega $. In fact, as shown in the proof of Theorem \ref{theo:bulk-relaxation}(i) (see also Remark \ref{Lipschitz regularity}) on each open subset $ U \subset \Omega $ we readily obtain 
\begin{align*}
  \liminf_{k \to \infty} \int_{\Omega \setminus D_k}W(\nabla y_k)\, \mathrm{d}x 
  &\geq \int_{\Omega \setminus D}W^{\rm qc}(\nabla y)\, \mathrm{d}x, \\ 
  \liminf_{k \to \infty}\psi(D \rchi_{D_k})(U) 
  &\geq 2  \int_{S_y \cap D^0 \cap U}\psi(\nu(y)) \, \mathrm{d}\mathcal{H}^{n-1}  + \psi(D \rchi_D)(U). 
\end{align*} 
Thus,  
\begin{align*}
	\limsup_{k \to \infty}\psi(D \rchi_{D_k})(\Omega) & \leq \limsup_{k \to \infty} \mathcal{E}(u_k, D_k) - \liminf_{k \to \infty} \int_{\Omega \setminus D_k}W(\nabla y_k)\, \mathrm{d}x \\
	& \leq  \mathcal{E}(u,D) - \int_{\Omega \setminus D}W^{\rm qc}(\nabla y)\, \mathrm{d}x \\
	& =  2  \int_{S_y \cap D^0}\psi(\nu(y)) \, \mathrm{d} \mathcal{H}^{n-1}  + \psi(D \rchi_D)(\Omega)
\end{align*} 
and the conclusion follows from the Portmanteau theorem.
\end{remark}

Since $\mathcal{E}^{\rm (rel)}(y, D) = \mathcal{E}^{\rm (rel)}(\rchi_{\Omega \setminus D} y, D)$, the following compactness statement complements Theorem~\ref{theo:bulk-relaxation}. 
\begin{proposition}\label{prop:bulk-compactness} 
If $p, q > 1$ and $(y_k, D_k) \subset W^{1,p}(\Omega; \R^{m}) \times \mathcal{F}(\Omega)$ is a sequence with 
\begin{align*}
	\mathcal{E}(y_k, D_k) + \| y_k \|_{L^q(\Omega; \R^m)} \le C 
\end{align*}
then there is a $(y,D) \in GSBV^p_q(\Omega; \R^n) \times \mathcal{F}(\Omega)$ such that $y =\rchi_{\Omega \setminus D} y $ and, for a subsequence, 
\begin{align*} 
	\rchi_{\Omega \setminus D_k} y_k \to  \rchi_{\Omega \setminus D} y \quad \mbox{in } L^1(\Omega;\R^m) \quad\mbox{and}\quad 
	\rchi_{D_k} \to \rchi_{D} \quad \mbox{in } L^1(\Omega). 
\end{align*}
\end{proposition} 

We will henceforth denote by $c, C > 0$ generic, $k$-independent constants whose value might change from one occurrence to the next. 

\begin{proof} 
This is a direct consequence of the Compactness Theorem \ref{theo:Ambrosio-compactness}: As $(D_k) \subset \mathcal{F}(\Omega)$ with $\mathcal{H}^{n-1}(\Omega \cap \partial^\ast D_k) \le C/\bar{c}$ we immediately get a subsequence (not relabeled) such that $\rchi_{D_k} \to \rchi_{D}$ in $L^1(\Omega)$ for some $D \in \mathcal{F}(\Omega)$. Moreover, 
$\rchi_{\Omega \setminus D_k} y_k \in SBV^p(\Omega; \mathbb{R}^m)$ and, by the growth condition \eqref{eq:p-growth},  
\begin{align*}
  &\bar{c}\int_{\Omega} | \nabla (\rchi_{\Omega \setminus D_k} y_k)|^p \, \mathrm{d}x + \bar{c}\mathcal{H}^{n-1}(S_{\rchi_{\Omega \setminus D_k} y_k}) + \| \rchi_{\Omega \setminus D_k} y_k \|_{L^q(\Omega;\R^m)} \\ 
  &~~\le  \mathcal{E}(y_k, D_k) + \| y_k \|_{L^q(\Omega;\R^m)} + \bar{C} \mathcal{L}^n(\Omega) \leq C, 
\end{align*} 
so that, after passing to a further subsequence (not relabeled), $\rchi_{\Omega \setminus D_k} y_k \to y$ in $L^1(\Omega;\R^m)$ for a $y \in GSBV^p_q(\Omega;\R^m)$, and in combination with $\rchi_{\Omega \setminus D_k} \to \rchi_{\Omega \setminus D}$ in $L^1(\Omega)$ and thus boundedly in measure, we get that also
\[ \rchi_{\Omega \setminus D_k} y_k 
   = \rchi_{\Omega \setminus D_k} \cdot \rchi_{\Omega \setminus D_k} y_k \to \rchi_{\Omega \setminus D} y \]
in $L^1(\Omega;\R^m)$. This proves the claim.
\end{proof}

\begin{remark}
\begin{enumerate}
\item Without any additional bound on $y_k$, the control of the energy alone is not sufficient to guarantee compactness in $GSBV$ as the specimen may fracture into pieces and mass might escape to infinity. Remarkably, in such a situation one can still obtain compactness modulo rigid motions subordinate to a Caccioppoli partition of the domain, cf.\ \cite{Friedrich:19}. 

\item In physical applications with a bounded region containing $y_k(\Omega)$ one has $\| y_k \|_{L^{\infty}(\Omega;\R^m)} \le C$. The energy bound $\mathcal{E}(y_k, D_k) \le C$ then leads to a limiting deformation $y \in SBV^p_{\infty}(\Omega;\R^m)$. 
\end{enumerate}
\end{remark}

\begin{proof}[Proof of Theorem~\ref{theo:bulk-relaxation}(i)] 
We can assume $ \sup_{k \in \mathbb{N}} \mathcal{E}(y_k, D_k)< \infty $. Let $(y, D) \in GSBV^p_1(\Omega; \R^{m}) \times \mathcal{F}(\Omega)$ and suppose that $(y_k, D_k) \subset W^{1,p}(\Omega; \R^{m}) \times \mathcal{F}(\Omega)$ is such that $y_k \to y$ in $L^1(\Omega;\R^m)$ and $\rchi_{D_k} \to \rchi_D$ in $L^1(\Omega)$. Note that $\rchi_{\Omega \setminus D} y \in GSBV^p_1(\Omega; \R^m)$ as for each component  $(\rchi_{\Omega \setminus D} y_i)^M  = \rchi_{\Omega \setminus D} y_i^M \in SBV(\Omega)$ for any truncation parameter $M > 0$ and $\nabla (\rchi_{\Omega \setminus D} y_i) = \rchi_{\Omega \setminus D} \nabla y_i$ $\mathcal{L}^n$ a.e., $i = 1, \ldots, n$, since $\mathcal{L}^n(\{|y_i| > M \}) \to 0$ as $M \to \infty$. 

By passing to $W - W^{\rm qc}(0)$ we may without loss of generality assume that $W^{\rm qc}(0) = 0$. We fix a constant $c_y \in \R^m$ such that 
\begin{align*} 
  &\mathcal{H}^{n-1} \big( \{ x \in \partial^* D : y^+(x) = c_y \} \big) \\ 
  &~~+ \mathcal{H}^{n-1} \big( \{ x \in S_y : \{y^{+}(x), y^{-}(x)\} \ni c_y \} \big) 
   = 0, 
\end{align*}
where $y^\pm$ denotes the traces of $y$ on $S_y$, respectively, $\partial^* D$ and, in particular, $y^{+}$ is the outer trace of $y$ on $\partial^* D$.  We also define the functions  $\tilde{y}_k \in SBV^p_1(\Omega; \R^m)$ by 
\[	 \tilde{y}_k 
	 = \rchi_{\Omega \setminus D_k} (y_k - c_y) \] 
so that $\tilde{y}_k \to \tilde{y} = \rchi_{\Omega \setminus D} (y - c_y)$ in $L^1(\Omega;\R^m)$ and $\mathcal{H}^{n-1}(S_{\tilde{y}_k}\setminus \partial^* D_k) = 0$. Then 
\begin{align*}
  \int_{\Omega \setminus D} W(\nabla y_k) \, \mathrm{d}x 
  \ge \int_{\Omega} W^{\rm qc}(\nabla \tilde{y}_k) \, \mathrm{d}x  
\end{align*}
and for any truncation parameter $M > 0$ 
\begin{align*}
   \psi(D \rchi_{D_k})(\Omega)
	&\ge \int_{S_{\tilde{y}_k}} g(\tilde{y}_k^+, \tilde{y}_k^-, \nu(\tilde{y}_k))  \, \mathrm{d}\mathcal{H}^{n-1} \\
	&= \int_{S_{\tilde{y}_k^M}} g \big((\tilde{y}_k^M)^+, (\tilde{y}_k^M)^-, \nu(\tilde{y}_k^M) \big)  \, \mathrm{d}\mathcal{H}^{n-1}, 
\end{align*}  
where 
\[ g(x, y, \nu) = 
	\begin{cases} 
	0 &\mbox{if } x = y, \\
	\psi(\nu) &\mbox{if } x \ne y \mbox{ and } |x| \cdot |y| = 0, \\ 
	2\psi(\nu) &\mbox{if } x \ne y \mbox{ and } |x| \cdot |y| \ne 0 
\end{cases} \] 
and $\tilde{y}_k^M$ denotes the function obtained from $\tilde{y}_k$ through a componentwise truncation. It is not hard to see that $g$ is jointly convex on $K \times K \times \R^n$ for any compact subset $K$ of $\R^m$. As, for each fixed $M$, $S_{\tilde{y}_k^M} \subset S_{\tilde{y}_k}$, $|\nabla \tilde{y}_k^M| \le |\nabla \tilde{y}_k|$ $\mathcal{L}^n$ a.e.\ and $\tilde{y}_k^M \to \tilde{y}^M$ in $ L^1(\Omega; \mathbb{R}^m)$, we infer from Theorem \ref{theo:Kristensen}, Theorem \ref{theo:BV-ellipticity} and the monotonicity of the set family $\{S_{\tilde{y}^M}: M >0\} $ that
	\begin{align*}
		\liminf_{k \to \infty} \mathcal{E}(y_k, D_k) 
		\ge \int_{\Omega} W^{\rm qc}(\nabla \tilde{y}) \, \mathrm{d}x 
		+ \int_{S_{\tilde{y}}} g\big(\tilde{y}^+, \tilde{y}^-, \nu(\tilde{y})\big)  \, \mathrm{d}\mathcal{H}^{n-1}. 
	\end{align*}
	Finally we observe that
	\begin{align*}
		&\int_{\Omega} W^{\rm qc}(\nabla \tilde{y}) \, \mathrm{d}x 
		+ \int_{S_{\tilde{y}}} g\big(\tilde{y}^+, \tilde{y}^-, \nu(\tilde{y})\big)  \, \mathrm{d}\mathcal{H}^{n-1} \\
		& \quad  = \int_{D^0 \cap \Omega} W^{\rm qc}(\nabla y) \, \mathrm{d}x 
		+ W^{\rm qc}(0) \mathcal{L}^n(D) \\ 
		& \qquad + 2  \int_{S_y \cap D^0} \psi(\nu(y))\, \mathrm{d}\mathcal{H}^{n-1} +  \int_{\Omega \cap \partial^\ast D} \psi(\nu(D))\, \mathrm{d}\mathcal{H}^{n-1}.  
\end{align*} 
\end{proof}

We now focus on the construction of the recovery sequence. We start with the following general fact on the asymptotic behaviour of the area of tubular neighbourhoods around certain sufficiently regular Borel sets. For every $E \subset \mathbb{R}^n$ and $r > 0$ we write $ \delta_E(x) = \inf \{ \psi^\circ(y-x): y \in E  \} $ for $ x \in \mathbb{R}^n $, 
\begin{align*}
 U_r(E) =  \{ x \in \mathbb{R}^n : \delta_E(x) < r   \}
\end{align*}
and note that $\partial^\ast U_r(E) = \partial^\ast \{ x \in \mathbb{R}^n : \delta_E(x) \geq r \} \subset \partial U_r(E)  $ for every $ r > 0 $ and that $ \delta_E $ is a Lipschitz function.  Combining the anisotropic coarea formula in \eqref{Coarea} with results on the outer Minkowski content, see \cite{LussardiVilla:16}, we prove in the next Lemma a key estimate for the construction of the recovery sequence. 
\begin{lemma}\label{area of distance sets}
Suppose $E \subset \mathbb{R}^n$ is a bounded Borel set such that $\partial E$ is $\mathcal{H}^{n-1}$ rectifiable and there exists $c > 0$ such that 
\begin{align*}
  \mathcal{H}^{n-1} (B_r(x) \cap \partial E) 
  \ge c r^{n-1} \qquad \textrm{for every $ x \in \partial E $ and $ r \in (0,1) $.}
\end{align*}
If $ B \subset \mathbb{R}^n  $ is a Borel set with $\mathcal{H}^{n-1}(\partial E \cap \partial B) =0$, then 
\begin{flalign}\label{eq:tubular-estimate 2}
 & \liminf_{r \to 0}\psi(D \rchi_{U_r(E)})(B) \notag  \\
 & \qquad  \le 2  \int_{B \cap \partial E \cap E^0} \psi(\nu(\partial E))\, \mathrm{d}\mathcal{H}^{n-1} +  \int_{B \cap \partial^\ast E} \psi(\nu(E))\, \mathrm{d}\mathcal{H}^{n-1}.
\end{flalign}
\end{lemma}

\begin{proof}
It follows from \cite[Theorem~4.4]{LussardiVilla:16} that 
\begin{flalign*}
	&	\lim_{r \to 0}\frac{\mathcal{L}^n(U_r(E) \setminus E)}{r} \notag  \\
	& \qquad  =  2  \int_{\partial E \cap E^0} \psi(\nu(\partial E))\, \mathrm{d}\mathcal{H}^{n-1} +  \int_{\partial^\ast E} \psi(\nu(E))\, \mathrm{d}\mathcal{H}^{n-1}.
\end{flalign*}
Moreover if $ A \subset \mathbb{R}^n $ is an open set, localizing the proof of the inequality \cite[(4.8)]{LussardiVilla:16} on $A$, we infer that 
\begin{flalign*}
	&	\liminf_{r \to 0}\frac{\mathcal{L}^n\big((U_r(E) \setminus E) \cap A  \big)}{r} \notag  \\
	& \qquad  \geq   2  \int_{A \cap \partial E \cap E^0} \psi(\nu(\partial E))\, \mathrm{d}\mathcal{H}^{n-1} +  \int_{A \cap \partial^\ast E} \psi(\nu(E))\, \mathrm{d}\mathcal{H}^{n-1}.
\end{flalign*}
Since the left hand side of this inequality defines a family of Radon measures $ \mu_r $ and the right hand side a Radon measure $ \mu $,  we infer from the Portmanteau theorem that $ \mu_r $ weakly* converges to $ \mu $ and, noting that $ \mu(\partial B) =0 $, we conclude 
\begin{flalign}\label{eq:tubular-estimate}
	&	\lim_{r \to 0}\frac{\mathcal{L}^n\big((U_r(E) \setminus E) \cap B  \big)}{r} \notag  \\
	& \qquad  =  2  \int_{B \cap \partial E \cap E^0} \psi(\nu(\partial E))\, \mathrm{d}\mathcal{H}^{n-1} +  \int_{B \cap \partial^\ast E} \psi(\nu(E))\, \mathrm{d}\mathcal{H}^{n-1} := c'.
\end{flalign}
We notice\footnote{Evidently it holds $ |\delta_E(x)- \delta_E(y)| \leq \psi^\circ(x-y) $ for all $x,y \in \mathbb{R}^n $, whence we deduce that $ \langle \nabla \delta_E(x), v \rangle \leq 1 $ if $ \delta_E $ is differentiable at $ x \in \mathbb{R}^n $ and $ \psi^\circ(v) = 1$.  Moreover if $ \delta_E $ is differentiable at $ x \in \mathbb{R}^n \setminus \overline{E} $ and $ a \in \overline{E} $ with $ \psi^\circ(x-a) = \delta_E(x) $, then $ \delta_E(x + t (a-x)) = \delta_E(x) - t \delta_E(x) $ and, differentiating this equality in $ t =0 $, we obtain $ \langle \nabla \delta_E(x), x-a\rangle = \delta_E(x) $. Therefore we conclude that $ \psi(\nabla \delta_E(x)) = 1 $ for $ \mathcal{L}^n $ a.e.\ $ x \in \mathbb{R}^n \setminus \overline{E} $.} that $ \psi(\nabla \delta_E(x)) = 1 $ for $ \mathcal{L}^n $ a.e.\ $ x \in \mathbb{R}^n \setminus \overline{E}$. Therefore we can use the anisotropic coarea formula in \eqref{Coarea} with $ \delta_E $ and $ (U_r(E)\setminus \overline{E}) \cap B $ to compute
\begin{align*}
  \frac{\mathcal{L}^n\big((U_r(E) \setminus E)\cap B  \big)}{r} &= \frac{1}{r} \int_{(U_r(E)\setminus \overline{E}) \cap B}\psi(\nabla \delta_E)\, \mathrm{d}x \\
  & = \frac{1}{r}  \int_{0}^{r}  \psi\big(D \rchi_{U_t(E)}\big)(B)\,\mathrm{d}t\\
& = \int_{0}^{1}  \psi\big(D \rchi_{U_{tr}(E)}\big)(B)\,\mathrm{d}t.
\end{align*}
We infer from Fatou's Lemma  that
\begin{align*}
	 \int_{0}^{1}\liminf_{r \to 0} \psi\big(D \rchi_{U_{tr}(E)}\big)(B)\, \mathrm{d}t \leq c'
\end{align*}
and consequently there exists $ t_0 \in (0,1) $ such that 
\[ \liminf_{r \to 0} \psi\big(D \rchi_{U_{t_0r}(E)}\big)(B)\leq c'. \qedhere \]
\end{proof}

\begin{remark} 
The condition $\mathcal{H}^{n-1}(\partial E) < \infty$ is not necessary. It is enough to assume that $\partial E$ is countably $\mathcal{H}^{n-1}$ rectifiable and  there exists $c > 0$ and a Radon measure $\eta$ on $\mathbb{R}^n$ absolutely continuous with respect to $\mathcal{H}^{n-1}$ such that $\eta(B_r(x)) \ge c r^{n-1}$ for every $x \in \partial E$ and $r \in (0,1)$; see \cite[Theorem~4.4 and Remark~4.2]{LussardiVilla:16}.
\end{remark}

For easy reference we also state the following well known relaxation result for $ W $ which directly follows from \cite[Theorem~9.1]{Dacorogna:08} by using \eqref{eq:Wqc-rep} instead of \cite[Theorem~6.9]{Dacorogna:08} in the proof of that theorem. 
\begin{lemma}\label{lem: relaxation}
	Let $ U \subseteq \mathbb{R}^n $ be an arbitrary open set and $ p \leq q \leq \infty $. For every $ u \in W^{1, q}(U, \mathbb{R}^m) $ there exists a sequence $ \varphi_k \in C^\infty_c(U, \mathbb{R}^m) $ such that $ \varphi_k \to 0 $ in $ L^q(U, \mathbb{R}^m) $ and 
	\begin{equation*}
	\lim_{k \to \infty}\int_{U} W(\nabla(u + \varphi_k))\, \mathrm{d}x = \int_{U}W^{\rm qc}(\nabla u)\, \mathrm{d}x.
	\end{equation*}
\end{lemma}

\begin{proof}[Proof of Theorem~\ref{theo:bulk-relaxation}(ii)] 
We proceed in consecutive steps. 
\smallskip 

\noindent{\em Step 1.} 
 Firstly we treat the case $y \in \mathcal{W}(\Omega; \mathbb{R}^m)$ (see Theorem \ref{theo:Cortesani}) and $ D \in \mathcal{C}(\Omega)$, where $D = A \cap \Omega$ with $A \subset \mathbb{R}^n$ open with smooth boundary such that $ \mathcal{H}^{n-1}(\partial A \cap \partial \Omega) =0 $. Suppose that $0 < c_k \le \mathcal{L}^n(\Omega)$ with $\lim_{k \to \infty} c_k = \mathcal{L}^n(D)$. We define $ E = A \cup S_y $ and we notice that $ E^0 = \mathbb{R}^n \setminus \overline{A} $, $ \partial^* E = \partial A $ and $ \partial E = \partial A \cup (\overline{S_y} \setminus \overline{A}) $ and $ \mathcal{H}^{n-1}(\partial E \cap \partial \Omega) =0 $. In particular it follows from Lemma \ref{area of distance sets} that there exists a positive  sequence $ \eta_k \to 0 $ such that 
	\begin{align*}
\lim_{k \to \infty} \psi\big(D \rchi_{U_{\eta_k}(E)}\big)(\overline{\Omega}) \leq c'
	\end{align*}
with $\int_{\Omega \cap \partial D}\psi(\nu(A))\, \mathrm{d}\mathcal{H}^{n-1} + 2 \int_{S_y \setminus \overline{D}} \psi(\nu(y))\, \mathrm{d}\mathcal{H}^{n-1} : = c' $. If $ D = \emptyset $, then  we need to select a further subsequence (not relabeled) of $(\eta_k) $ such that $ \eta_k \ll c_k $. Notice that $ \rchi_{U_{\eta_k}(E)} \to \rchi_E $ in $ L^1(\mathbb{R}^n) $.   Approximating the sets $U_{\eta_k}(E) $ with smooth sets by means of Theorem \ref{theo:ComiTorres}, we find a sequence $(E_k)$ of open sets in $ \mathbb{R}^n $ with smooth boundaries such that $ \mathcal{H}^{n-1}(\partial E_k \cap \partial \Omega) =0 $ for every $ k \geq 1 $, $ \rchi_{E_k} \to \rchi_{E} $ in $ L^1(\mathbb{R}^n) $,  $ U_{\eta_k/2}(E) \subset E_k  \subset U_{3\eta_k/2}(E) $ and
\begin{equation}\label{eq:Dktilde-perim-est}
\lim_{k \to \infty} \psi\big(D \rchi_{E_k}\big)(\overline{\Omega}) \leq c'.
\end{equation} 
It follows that
\begin{align}\label{eq:Dktilde-vol-est}
  \mathcal{L}^n(E_k \cap \Omega) = \mathcal{L}^n(D) + O(\eta_k).
\end{align}
Now, if $D \ne \Omega$, fix $x_0 \in \Omega \setminus \overline{E}$ and,  if $D \ne \emptyset$, $y_0 \in D$. Then, noting that  $ \mathcal{L}^n(E_k \cap \Omega) - c_k \to 0 $, we set 
\[ \tilde{E}_k 
   = \big( E_k \cup B_{\sigma_k}(x_0) \big) \setminus \overline{B_{\tau_k}(y_0)} \] 
with null sequences $ \sigma_k, \tau_k \geq 0$ such that $ \overline{B_{\sigma_k}(x_0)} \subset \Omega \setminus \overline{E} $, $ \overline{B_{2\tau_k}(y_0)} \subset D $ and
\begin{align*}
\mathcal{L}^n(E_k \cap \Omega) - c_k =  \mathcal{L}^n(B_{1}) \tau_k^n -  \mathcal{L}^n(B_{1}) \sigma_k^n \quad \textrm{for all $ k $ sufficiently large.}
\end{align*}
If $ D = \emptyset $ this choice is possible with $\tau_k =0$ (and $B_{\tau_k}(y_0) = \emptyset$) for every $k$ as in this case $\mathcal{L}^n(E_k \cap \Omega) - c_k = O(\eta_k)- c_k $  by \eqref{eq:Dktilde-vol-est} and $ \eta_k \ll c_k $. In case $ D = \Omega $ we understand that  $B_{\sigma_k}(x_0) = \emptyset$ and use that $c_k \le \mathcal{L}^n(\Omega)$ by assumption. The modified sets $\tilde{E}_k$ still satisfy
\begin{align}\label{eq:Dk-perim-est}
\lim_{k \to \infty} \psi\big(D \rchi_{\tilde{E}_k}\big)(\overline{\Omega}) \leq c'
    \end{align}
by \eqref{eq:Dktilde-perim-est} and 
$ \rchi_{\tilde{E}_k} \to \rchi_E $ in $L^1(\mathbb{R}^n)$  with
\[
   \mathcal{L}^n(\tilde{E}_k \cap \Omega) 
   = \mathcal{L}^n(E_k \cap \Omega) + \alpha_n \sigma_k^n - \alpha_n \tau_k^n = c_k\]
for $k$ sufficiently large enough, as then $\overline{B_{\sigma_k}(x_0)} \cap E_k = \emptyset$ and $\overline{B_{\tau_k}(y_0)} \subset E_k$.

Since $ y \in W^{l, \infty}(\Omega \setminus \overline{S_y}, \mathbb{R}^m) $ for every $ l \geq 1 $, we can apply the relaxation result in Lemma \ref{lem: relaxation} to find a sequence  $\varphi_k \in C^{\infty}(\mathbb{R}^n; \R^m)$ such that $ \spt \varphi_k \subset \Omega \setminus \overline{E_k} $, $ \|\varphi_k \|_{L^\infty(\Omega \setminus \overline{E_k}, \mathbb{R}^m)} \leq \frac{1}{k}$ and 
\begin{align*}
  \int_{\Omega \setminus \overline{E_k}} W (\nabla(y + \varphi_k)) \, \mathrm{d} x 
  \le \int_{\Omega \setminus \overline{E_k}} W^{\rm qc} (\nabla y) \, \mathrm{d} x + \frac{1}{k}. 
\end{align*} 
 For each $ k \geq 1 $ we select cut-off functions $ \theta_k \in C^\infty(\mathbb{R}^n) $ such that $ 0 \leq \theta_k \leq 1 $, $ \theta_k = 1 $ on $ \Omega \setminus \overline{E_k} $, $ \theta_k =0 $ on a neighbourhood of $ J_y $  and $ \rchi_{\{\theta_k < 1\}} \to 0 $ in $ L^1(\mathbb{R}^n) $ as $ k \to \infty $. We define $ \tilde{y}_k = \theta_k(y + \varphi_k) $ and we notice that $ \tilde{y}_k \in W^{l, \infty}(\Omega, \mathbb{R}^m) $ for every $ l \geq 1 $, $ \tilde{y}_k \to y $ in $ L^1(\Omega) $ and 
\begin{equation*}
\int_{\Omega \setminus E_k} W (\nabla \tilde{y}_k) \, \mathrm{d} x 
\le \int_{\Omega \setminus E_k} W^{\rm qc} (\nabla y) \, \mathrm{d} x + \frac{1}{k}.
\end{equation*}
Using Stein's extension theorem we obtain that $ \tilde{y}_k \in C^\infty(\overline{\Omega}, \mathbb{R}^m) $; see \cite[Theorem~5 in Chap.~VI]{Stein:70}.
Choosing $y_k = \zeta_k \tilde{y}_k$, where $\zeta_k \in C^{\infty}(\R^n)$ is a cut-off function with $0 \le \zeta_k \le 1$ on $\R^n$ and $\zeta_k = 0$ on $B_{\tau_k}(y_0)$ and $\zeta_k = 1$ on $\R^n \setminus B_{2\tau_k}(y_0)$,  and also using \eqref{eq:p-growth} and \eqref{eq:Dktilde-vol-est}, we obtain 
\begin{align*}
  \int_{\Omega \setminus \tilde{E}_k} W (\nabla y_k) \, \mathrm{d} x 
  &\le \int_{\Omega \setminus E_k} W(\nabla \tilde{y}_k) \, \mathrm{d} x  + C \sigma_k + \mathcal{L}^n(B_{\tau_k}(y_0)) W(0) \notag \\ 
  &\le \int_{\Omega \setminus D} W^{\rm qc} (\nabla y) \, \mathrm{d} x + \frac{1}{k} + C (\sigma_k + \eta_k) + \mathcal{L}^n(B_{\tau_k}(y_0)) W(0), 
\end{align*} 
while still $y_k \to y$ in $L^1(\Omega; \R^m)$. Combining with \eqref{eq:Dk-perim-est}  and setting $D_k = \tilde{E}_k \cap \Omega$ we find that indeed 
\[ \limsup_{k \to \infty} \mathcal{E}(y_k, D_k) \le \mathcal{E}^{\rm rel}(y, D), \]
which concludes the proof of Step 1. 
\smallskip 

\noindent{\em Step 2.} 
Next we consider $(y, D) \in  GSBV^p_1(\Omega;\R^m) \times \mathcal{C}(\Omega)$, say $D = A \cap \Omega$ with a bounded set $A$ with smooth boundary such that $ \mathcal{H}^{n-1}(\partial A \cap \partial \Omega) =0 $. 

We notice that for the functions $y^k = (y_1^k, \ldots, y_n^k)$ with cut-off components one has $ y^k \in SBV^p_\infty(\Omega; \mathbb{R}^m) $, $y^k \to y$ in $L^1(\Omega; \R^m)$ and, according to Section~\ref{section:BV}, $S_y = \bigcup_{k = 1}^\infty S_{y^k}$, where the family $ k \mapsto S_{y^k} $ is increasing, and $\nabla y^k(x) \to \nabla y(x)$ for $\mathcal{L}^n$ a.e.\ $ x \in \Omega $, where $|W^{\rm qc}(\nabla y^k)| \le C |\nabla y^k|^p + C \le C |\nabla y|^p + C \in L^1(\Omega; \R^{m \times n})$ due to \eqref{eq:p-growth}. By monotone and dominated convergence we infer that 
\[ \lim_{k \to \infty} \mathcal{E}^{\rm rel}(y^k, D) 
= \mathcal{E}^{\rm rel}(y, D). \] 
Therefore we can assume $ y \in SBV^p_\infty(\Omega; \mathbb{R}^m) $ and in view of Step~1, by invoking a diagonal sequence argument, it is now sufficient to provide a sequence $(y_k, D_k) \subset \mathcal{W}(\Omega; \mathbb{R}^m) \times \mathcal{C}(\Omega)$ such that $(y_k, \rchi_{D_k}) \to (y, \rchi_D) $ in $L^1(\Omega; \R^m) \times L^1(\Omega)$ with $\mathcal{L}^n(D_k) = \mathcal{L}^n(D) $ for each $ k \geq 1 $ and
\begin{align}\label{eq:step-two-upper} 
  \limsup_{k \to \infty} \mathcal{E}^{\rm rel}(y_k, D_k) 
  \le \mathcal{E}^{\rm rel}(y, D). 
\end{align} 
We choose approximations $y_k \to y$ in $L^1(\Omega; \R^m)$ as in Theorem~\ref{theo:Cortesani}. If $D =  \emptyset$ we immediately obtain \eqref{eq:step-two-upper} with $D_k = D = \emptyset$. In case $D \ne \emptyset$, arguing as in Step 1 with the help of Lemma \ref{area of distance sets} and Theorem \ref{theo:ComiTorres} we find a positive sequence $ \eta_k \to 0 $ and a sequence $ (A_k) $ of open subsets of $ \mathbb{R}^n $ with smooth boundaries such that $ \mathcal{H}^{n-1}(\partial A_k \cap \partial \Omega) =0 $ for every $ k \geq 1 $, $ \rchi_{A_k} \to \rchi_A $ in $ L^1(\mathbb{R}^n) $, $ U_{\eta_k/2}(A) \subset A_k \subset U_{3\eta_k/2}(A) $ and
\begin{equation*}
	\lim_{k \to \infty}\psi(D \rchi_{A_k})(\overline{\Omega})\leq \psi(D \rchi_A)(\Omega).
\end{equation*}
We choose $ y_0 \in D $ and $ \tau_k > 0 $ so that $\alpha_n \tau_k^n = \mathcal{L}^n( (A_k \cap \Omega)\setminus D)$. We set 
\begin{equation*}
D_k = (A_k \cap \Omega)\setminus \overline{B_{\tau_k}(y_0)}.
\end{equation*}
Note that $\overline{B_{\tau_k}(y_0)} \subset D$ if $k$ is large enough and then $\mathcal{L}^n(D_k) = \mathcal{L}^n(D)$. We have 
\[ \lim_{k \to \infty} \lim_{h \to \infty} \int_{\Omega \setminus D_k} W^{\rm qc} (\nabla y_h) \, \mathrm{d} x 
   = \lim_{k \to \infty} \int_{\Omega \setminus D_k} W^{\rm qc} (\nabla y) \, \mathrm{d} x 
   = \int_{\Omega \setminus D} W^{\rm qc} (\nabla y) \, \mathrm{d} x \] 
by Theorem~\ref{theo:Cortesani}(ii) and the dominated convergence theorem; moreover
\begin{align*}
  &\limsup_{k \to \infty} \limsup_{h \to \infty} \int_{S_{y_h}  \setminus \overline{D_k}} \psi(\nu(y_h))\, \mathrm{d}\mathcal{H}^{n-1}
  \le \limsup_{k \to \infty} \limsup_{h \to \infty} \int_{S_{y_h}  \setminus D_k} \psi(\nu(y_h))\, \mathrm{d}\mathcal{H}^{n-1} \\
  &~~\le \limsup_{k \to \infty} \int_{S_{y}  \setminus D_k} \psi(\nu(y))\, \mathrm{d}\mathcal{H}^{n-1} 
  = \int_{S_{y}  \setminus \overline{D}} \psi(\nu(y))\, \mathrm{d}\mathcal{H}^{n-1}
\end{align*}
by Theorem~\ref{theo:Cortesani}(iii') and dominated convergence as  $\rchi_{D_k}(x) \to \rchi_{\overline{D}}(x)$ for all $ x \in \Omega \setminus \{y_0\} $. We conclude that
\[ \limsup_{k \to \infty}\limsup_{h \to \infty} \mathcal{E}^{\rm rel}(y_h, D_k) 
   \le \mathcal{E}^{\rm rel}(y, D), \] 
   whence we can easily obtain the sequence satisfying \eqref{eq:step-two-upper}.
\smallskip

\noindent{\em Step 3.} 
In this crucial step we consider the case $(y, D) \in GSBV^p_1(\Omega;\R^m) \times \mathcal{F}(\Omega)$ such that both $D$ and $\Omega \setminus D$ have nonempty interior. We have to provide a sequence $D_k \in \mathcal{C}(\Omega)$ with $\mathcal{L}^n(D_k) = \mathcal{L}^n(D)$ such that 
\[ \lim_{k \to \infty} \mathcal{E}^{\rm rel}(y, D_k) 
   = \mathcal{E}^{\rm rel}(y, D). \]

To this end, we first extend the set $D$ to a set $E$ with $E \cap \Omega = D$ in such a way that $\mathcal{H}^{n-1}(\partial^* E \cap \partial \Omega) = 0$ following the procedure outlined in \cite[Remark 3.43]{AmbrosioFuscoPallara:00}: since $\Omega$ is a Lipschitz domain, the function $\rchi_D$ can be extended to a mapping $f \in BV(\R^n)$ with compact support such that $|Df|(\partial \Omega) = 0$. We note that for a.e.\ choice of $s \in (0,1)$ the set $F_s = \{ x \in \R^n : f(x) > s \}$ is of finite perimeter and satisfies $\mathcal{H}^{n-1} (\partial^* F_s \cap \partial \Omega) = |D \rchi_{F_s}|(\partial \Omega) = 0$ by the coarea formula $|Df|(\partial \Omega) = \int_{-\infty}^{\infty} |D \rchi_{F_s}|(\partial \Omega) \, \mathrm{d}s$. Fixing such an $s$ we set $E = F_s$. 

We then choose an approximating sequence $E_k$ for $E$ as in Theorem~\ref{theo:ComiTorres} such that  $\mathcal{H}^{n-1}(\partial E_k \cap \partial \Omega) =0$ for every $ k \in \N $. We fix $x_0 \in \into (\Omega \setminus D)$, $y_0 \in \into D$ and set  $ \rho = \min\{ \dist(x_0, (\R^n \setminus \Omega) \cup E), \dist(y_0, \mathbb{R}^n \setminus D)\} $ and
\[ D_k 
   = \big( (E_k \cap \Omega) \cup B_{\sigma_k}(x_0) \big) \setminus \overline{B_{\tau_k}(y_0)} \in \mathcal{C}(\Omega),
\]
where  $0 \le \sigma_k, \tau_k \le \rho/2$ with $\sigma_k \to 0$, $\tau_k \to 0$ are chosen such that $\alpha_n(\sigma_k^n - \tau_k^n) = \mathcal{L}^n(D) - \mathcal{L}^n(E_k \cap \Omega)$ for large $k$. Since $B_{\sigma_k}(x_0) \cap E_k = \emptyset$ and $B_{\tau_k}(y_0) \subset E_k \cap \Omega$ for $ k $ sufficiently large by Theorem~\ref{theo:ComiTorres}(v)-(vi), this in particular yields $\mathcal{L}^n(D_k) = \mathcal{L}^n(D)$ for all $ k  $ large enough. 

By Theorem~\ref{theo:ComiTorres}(i)  and dominated convergence we have 
\[ \lim_{k \to \infty} \int_{\Omega \setminus D_k} W^{\rm qc} (\nabla y) \, \mathrm{d} x 
   = \int_{\Omega \setminus D} W^{\rm qc} (\nabla y) \, \mathrm{d} x. \] 
Since $\mathcal{H}^{n-1} (\partial^* E \cap \partial \Omega) = 0$, then $ \psi(D \rchi_E)(\partial \Omega) =0 $ and from Theorem~\ref{theo:ComiTorres}(ii) we get 
\[ \lim_{k \to \infty} \psi(D \rchi_{E_k})(\Omega) = \psi(D \rchi_E)(\Omega).\] 
For the remaining term we apply Theorem~\ref{theo:ComiTorres}(iii) with $\mu = \int_{S_y \cap \, \cdot} \psi(\nu(y))\, \mathrm{d}\mathcal{H}^{n-1}$ to get 
\[ \int_{S_y \cap [(D^1 \cup \partial^\ast D) \triangle \overline{D_k}]} \psi(\nu(y))\, \mathrm{d}\mathcal{H}^{n-1} = 0. \] 
Using $D_k^0 \cap \Omega = \Omega \setminus \overline{D_k}$ and $\mathcal{H}^{n-1} \big( \Omega \setminus (D^1 \cup \partial^* D \cup D^0) \big) = 0$, we thus obtain 
\begin{align*} 
 \int_{S_y \cap D^0_k} \psi(\nu(y))\, \mathrm{d}\mathcal{H}^{n-1} 
  & =  \int_{S_y} \psi(\nu(y))\, \mathrm{d}\mathcal{H}^{n-1} -  \int_{S_y \cap \overline{D_k}} \psi(\nu(y))\, \mathrm{d}\mathcal{H}^{n-1} \\ 
  & = \int_{S_y} \psi(\nu(y))\, \mathrm{d}\mathcal{H}^{n-1} -  \int_{S_y \cap  (D^1 \cup \partial^\ast D)} \psi(\nu(y))\, \mathrm{d}\mathcal{H}^{n-1} + o(k)\\
  & = \int_{S_y \cap D^0} \psi(\nu(y))\, \mathrm{d}\mathcal{H}^{n-1} + o(k) 
\end{align*}
as $k \to \infty$. 
\smallskip

\noindent{\em Step 4.} 
For the general case $(y, D) \in GSBV^p_1(\Omega;\R^m) \times \mathcal{F}(\Omega)$ it is now sufficient to provide a sequence $(D_k )\subset \mathcal{F}(\Omega)$ such that $\mathcal{L}^n(D_k) = \mathcal{L}^n(D)$, $\into D_k \ne \emptyset$, $\into (\Omega \setminus D_k) \ne \emptyset$ and 
\[ \lim_{k \to \infty} \mathcal{E}^{\rm rel}(y, D_k) 
   = \mathcal{E}^{\rm rel}(y, D). \] 
Note that by Step 1 we may assume without loss of generality $0 < \mathcal{L}^n(D) < \mathcal{L}^n(\Omega)$. Fix $x_0 \in D^0 \cap \Omega$ and $y_0 \in D^1$ and notice that 
\begin{align*}
	\mathcal{L}^n\big(B_\rho(x_0) \setminus D\big) >0 \quad \textrm{and} \quad \mathcal{L}^n\big(B_{\rho}(y_0) \cap D\big) >0
\end{align*} 
for all $\rho > 0$. Therefore we can choose two sequences $ \sigma_k \searrow 0 $ and $ \tau_k \searrow 0 $ such that 
\begin{align*}
	\mathcal{L}^n\big(B_{\sigma_k}(x_0) \setminus D\big) = \mathcal{L}^n\big(B_{\tau_k}(y_0) \cap D\big)
\end{align*}
and define
\[ D_k 
   = (D \cup B_{\sigma_k}(x_0) ) \setminus \overline{B_{\tau_k}(y_0)}. \]
These sets evidently satisfy  $D_k \in \mathcal{F}(\Omega)$ with $\mathcal{L}^n(D_k) = \mathcal{L}^n(D)$ for large $k$. Moreover, we have
\[ \lim_{k \to \infty} \mathcal{E}^{\rm rel}(y, D_k) 
   = \mathcal{E}^{\rm rel}(y, D). \qedhere \] 
\end{proof}

\section{Membrane limits}

In this section we fix $ p > 1 $, $ \omega \subset \mathbb{R}^2 $ a bounded open set with Lipschitz boundary and set $ \Omega_k= \omega \times (0, h_k) \subset \mathbb{R}^3 $ for a sequence of `membrane heights' $h_k > 0$ with $h_k \to 0$. Also set $\Omega =\omega \times (0, 1)$. 
To simplify the notation we make the following identifications: 
\begin{align*}
  GSBV^p(\omega; \mathbb{R}^3) 
  &= \big\{ u \in GSBV^p(\Omega; \mathbb{R}^3) : \partial_3 u =0, \; \nu_3(u)  =0 \big\}, \\ 
  \mathcal{F}(\omega) 
  &= \big\{ D \in \mathcal{F}(\Omega)  : \nu_3(D) =0 \big\}.
\end{align*}
Our goal is to analyse the asymptotic behaviour of the energy
\begin{align*}
	\mathcal{G}_k: W^{1,p}(\Omega_k; \mathbb{R}^3) \times \mathcal{F}(\Omega_k) \rightarrow \mathbb{R}
\end{align*} 
given by
\begin{align*}
	\mathcal{G}_k(u,D) = \int_{\Omega_k \setminus D} W(\nabla u) \, \mathrm{d}x +  \int_{\Omega_k \cap \partial^\ast D} \psi(\nu(D))\, \mathrm{d}\mathcal{H}^{2}
\end{align*}
for $(u, D) \in W^{1,p}(\Omega_k; \mathbb{R}^3)  \times \mathcal{F}(\Omega_k) $. The function $ W : \mathbb{R}^{3 \times 3} \rightarrow \mathbb{R} $ is a continuous function and $ \psi $ is an arbitrary norm on $ \mathbb{R}^3 $. We assume that they both satisfy the growth assumptions of the previous section.  To set the problem in the fixed domain $ \Omega $ we apply the usual rescaling $ \Omega_k \ni (x', x_3) \mapsto (x',  h_k^{-1} x_3) \in \Omega $.  (We use a prime to denote the first two columns or entries of a matrix, respectively, vector and in particular write $ \nabla u =(\nabla'u, \partial_3u) $.) We then consider the rescaled functional
\begin{align*}
	\mathcal{E}_{k} : W^{1,p}(\Omega; \mathbb{R}^3) \times \mathcal{F}(\Omega) \rightarrow  \mathbb{R}
\end{align*}
given by
\begin{align*}\label{rescaling}
	\mathcal{E}_{k}(u,D) = \int_{\Omega \setminus D}W\big(\nabla' u, h_k^{-1} \partial_3 u\big)\, \mathrm{d}x  + \int_{\Omega \cap \partial^* D} \psi\big(\nu'(D), h_k^{-1} \nu_3(D)\big)\, \mathrm{d}\mathcal{H}^2 
\end{align*}
for $(u, D) \in W^{1,p}(\Omega; \mathbb{R}^3) \times \mathcal{F}(\Omega)$.

We define the `membrane functional' $ \mathcal{E}_0 : W^{1,p}(\omega; \mathbb{R}^3) \times \mathcal{F}(\omega) \rightarrow \mathbb{R} $ by
\begin{align*}
	\mathcal{E}_0(u,D)= \int_{\omega \setminus D}W_0(\nabla' u)\, \mathrm{d}x' +  \int_{\omega \cap \partial^\ast D} \psi_0(\nu'(D))\, \mathrm{d}\mathcal{H}^{1} 
\end{align*}
for all $(u,D)\in W^{1,p}(\omega; \mathbb{R}^3) \times \mathcal{F}(\omega) $, where
\begin{equation}\label{eq: W0}
	W_0(\xi') = \inf \big\{ W(\xi', \xi_3) : \xi_3 \in \mathbb{R}^3  \big\} \quad \textrm{for all $ \xi' \in \mathbb{R}^{3 \times 2} $} 
\end{equation}
and 
\begin{equation}\label{eq: psi0}
	\psi_0(v') = \inf\{ \psi(v', v_3) : v_3 \in \mathbb{R} \}\quad \textrm{for all $ v' \in \mathbb{R}^2 $.}
\end{equation}
The relaxation of $ \mathcal{E}_0 $ is given by the functional $ \mathcal{E}_0^{\rm rel} : GSBV^p_1(\omega; \mathbb{R}^3) \times \mathcal{F}(\omega)\rightarrow \mathbb{R}  $,
\begin{align*}
\mathcal{E}_0^{\rm rel}(u,D) 
= \int_{\omega \setminus D}W_0^{\rm qc}(\nabla' u)\, \mathrm{d}x 
+ 2\int_{S_u \cap D^0} \psi_0(\nu'(u))\, \mathrm{d}\mathcal{H}^{1}
+  \int_{\omega \cap \partial^\ast D} \psi_0(\nu'(D))\, \mathrm{d}\mathcal{H}^{1}. 
\end{align*}
We remark that $W_0$ and thus also $W_0^{\rm qc}$ satisfy the same growth condition \eqref{eq:p-growth} as $W$ (with $X$ replaced by $X' \in \R^{3 \times 2}$). Moreover $ \psi_0 $ is a norm on $ \mathbb{R}^2 $ with $ \bar{c}|v'| \leq \psi_0(v') \leq \bar{C}|v'| $ for each $ v' \in \mathbb{R}^2 $. Our main theorem for membranes is the following $\Gamma(L^1)$-convergence result with possible volume constraints.

\begin{theorem}\label{theo:membrane-Gamma} 
The functionals $\mathcal{E}_{k}$ $\Gamma(L^1)$-converge to $\mathcal{E}_0^{\rm rel}$, in fact:  
\begin{itemize} 
\item[(i)] whenever $(u_k) \subset W^{1,p}(\Omega; \R^{3})$ and $(D_k) \subset \mathcal{F}(\Omega)$ are such that $u_k \to u$ in $L^1(\Omega; \R^{3})$ and $\rchi_{D_k} \to \rchi_{D}$ in $L^1(\Omega)$ for some $u \in GSBV^p_1(\omega; \R^{m})$ and $D \in \mathcal{F}(\omega)$, then one has 
\[ \liminf_{k \to \infty} \mathcal{E}_k(u_k, D_k) 
   \ge \mathcal{E}_0^{\rm rel}(u, D), \] 
\item[(ii)] for each $(u, D) \in  GSBV^p_1(\omega; \R^{3}) \times \mathcal{F}(\omega)$ and $c_1, c_2, \ldots \in (0,\mathcal{L}^2(\omega)]$ with $c_k \to \mathcal{L}^2(D)$ there are $(u_k) \subset C^{\infty}(\overline{\Omega}; \R^3)$ and $(D_k) \subset \mathcal{C}(\Omega)$ with $u_k \to u$ in $L^1(\Omega; \R^{3})$, $\rchi_{D_k} \to \rchi_{D \times (0,1)}$ in $L^1(\Omega)$ and $\mathcal{L}^3(D_k) = c_k$ for all $k \in \N$  and
\[ \lim_{k \to \infty} \mathcal{E}_k(u_k, D_k) 
= \mathcal{E}_{0}^{\rm rel}(u, D). \]
\end{itemize}
\end{theorem}

As for the bulk model, firstly we notice the following compactness property, which essentially follows from Proposition~\ref{prop:bulk-compactness}. 
\begin{proposition}\label{Membranes compactness}
	If $ q > 1 $ and $(u_k, D_k) \subset W^{1,p}(\Omega; \mathbb{R}^3) \times \mathcal{F}(\Omega) $ is a sequence with 
	\begin{align*}
		\mathcal{E}_k(u_k, D_k) + \| u_k \|_{L^q(\Omega;\R^3)} \leq C,
	\end{align*}
	then there exists $(u,D)\in GSBV^p_q(\omega; \mathbb{R}^3) \times \mathcal{F}(\omega)$ such that $ u = \rchi_{\Omega \setminus D} u$ and , for a subsequence,
	\begin{align*}
		\rchi_{\Omega \setminus D_k} u_k \to  \rchi_{\Omega \setminus D} u \quad \textrm{in $L^1(\Omega; \mathbb{R}^3) $} \qquad \textrm{and} \qquad  \rchi_{D_k} \to\rchi_{D} \quad \textrm{in $L^1(\Omega) $}.
	\end{align*}
\end{proposition}

\begin{remark}
Both in Theorem~\ref{theo:membrane-Gamma} and in Proposition~\ref{Membranes compactness} suitable loading terms can be included, e.g., given by some $\ell \in L^{p'}(\Omega; \R^3)$, $p' = p/(p-1)$, acting on the elastic part of the body. Setting 
\begin{align*} 
  L : W^{1,p}(\Omega; \mathbb{R}^3) \times \mathcal{F}(\Omega) \to \R, \quad 
  L(u, D) 
   = \int_{\Omega \setminus D} \ell \cdot u \, \mathrm{d}x, \\ 
  \bar{L} : GSBV^p_1(\omega; \mathbb{R}^3) \times \mathcal{F}(\omega) \to \R, \quad 
   \bar{L}(u, D) 
   = \int_{\omega \setminus D} \ell \cdot u \, \mathrm{d}x', 
\end{align*}
where $\bar{\ell}(x') = \int_0^1 \ell(x', x_3) \, \mathrm{d} x_3$, a standard argument shows that for a sequence of almost minimizers $(u_k, D_k)$ of $\mathcal{E}_k + L$ with uniformly bounded $\| u_k \|_{L^q(\Omega;\R^3)}$ one has $\rchi_{\Omega \setminus D_k} u_k \to u$ in $L^1(\Omega; \mathbb{R}^3) $ and $\rchi_{D_k} \to\rchi_{D}$ in $L^1(\Omega) $ for a minimizer $(u, D)$ of $\mathcal{E}_0^{\rm rel} + \bar{L}$. 
\end{remark}

\begin{proof}[Proof of Proposition~\ref{Membranes compactness}] 
Since $W(X', h_k^{-1}X_3) \ge \bar{c} |(X', h_k^{-1}X_3)|^p - \bar{C} \ge \bar{c} |X|^p - \bar{C}$ and $\psi(\nu', h_k^{-1}\nu_3) \geq \bar{c} |(\nu', h_k^{-1}\nu_3)| \geq \bar{c} $, proceeding as in the proof of Proposition~\ref{prop:bulk-compactness} we immediately obtain a subsequence (not relabeled) such that $ \rchi_{D_k} \to \rchi_{D} $ in $ L^1(\Omega) $ and $\rchi_{\Omega \setminus D_k} u_k \to u $  in $ L^1(\Omega; \mathbb{R}^3) $ for some $ u \in GSBV^p_1(\Omega; \mathbb{R}^3) $ and $ D \in \mathcal{F}(\Omega) $ with $ u = \rchi_{\Omega \setminus D} u $. 

It remains to show that $ \nu_3(D)=0 $, $ \partial_3 u =0 $ and $ \nu_3(u)=0 $. Applying Theorem \ref{theo:BV-ellipticity} we infer that for every $M > 0$ 
\begin{align*}
  M \bar{c} \int_{\Omega \cap \partial^* D}|\nu_3(D)| \, \mathrm{d}\mathcal{H}^2 & \leq c \int_{\Omega \cap \partial^* D}|(\nu'(D), M \nu_3(D))| \, \mathrm{d}\mathcal{H}^2 \\
  & \leq \bar{c} \liminf_{k \to \infty}  \int_{\Omega \cap \partial^* D_k}|(\nu'(D_k), M \nu_3(D_k))| \, \mathrm{d}\mathcal{H}^2\\
  & \leq \bar{c} \liminf_{k \to \infty}  \int_{\Omega \cap \partial^* D_k}|(\nu'(D_k), h_k^{-1} \nu_3(D_k))| \, \mathrm{d}\mathcal{H}^2\\
  & \leq \liminf_{k \to \infty} \mathcal{E}_k(u_k, D_k) 
  \leq C
\end{align*}
and, setting $v_k = \rchi_{\Omega \setminus D_k} u_k$ and denoting by $v_k^M$ its componentwise truncation, likewise 
\begin{align*}
  M \bar{c} \int_{S_{u^M}} |\nu_3(u^M)| \, \mathrm{d}\mathcal{H}^2 
  & \leq \bar{c} \liminf_{k \to \infty}  \int_{S_{v_k^M}} |(\nu'(v_k^M), M \nu_3(v_k^M))| \, \mathrm{d}\mathcal{H}^2 \\ 
  & \leq \bar{c} \liminf_{k \to \infty}\int_{\Omega \cap \partial^* D_k}|(\nu'(D_k), M \nu_3(D_k))| \, \mathrm{d}\mathcal{H}^2 
  \le C. 
\end{align*}
Similarly, Theorem~\ref{theo:Kristensen} gives 
\begin{align*}
  M^p \bar{c} \int_{\Omega}|\partial_3 u|^p\,\mathrm{d}x 
  & \leq \bar{c} \liminf_{k \to \infty} \int_{\Omega}|(\nabla' v_k, M \partial_3 v_k)|^p\,\mathrm{d}x \\
  & \leq \liminf_{k \to \infty} \mathcal{E}_k(u_k, D_k) + \bar{C} \mathcal{L}^n(\Omega) 
   \le C. 
\end{align*}
The desired conclusion follows by sending $M$ to $\infty$.  
\end{proof}

We now prove the lower bound in Theorem~\ref{theo:membrane-Gamma}.  

\begin{proof}[Proof of Theorem~\ref{theo:membrane-Gamma}(i)] 
 We may assume $\sup_{k \in \mathbb{N}} \mathcal{E}_k(u_k, D_k) < \infty$. For $\eta \ge 0$ we define the auxiliary bulk functionals $ \mathcal{E}_0^\eta : W^{1,p}(\Omega; \mathbb{R}^3) \times \mathcal{F}(\Omega) \to \mathbb{R} $ by
\begin{align*}
  \mathcal{E}_0^\eta(u,D) 
  = \int_{\Omega \setminus D}W_0^\eta(\nabla u) \, \mathrm{d}x + \int_{\Omega \cap \partial^\ast D} \psi_0^\eta(\nu(D))\, \mathrm{d}\mathcal{H}^{2}, 
\end{align*}
where $W_0^\eta(X) = W_0(X')+ \eta|X_{3}|^p$ and $ \psi_0^\eta(v) = \psi_0(v') + \eta |v_3| $, so that 
\begin{flalign}\label{eq:lb-est-mem}
  \liminf_{k \to \infty} \mathcal{E}_k(u_k, D_k) & \ge \liminf_{k \to \infty} \mathcal{E}_0^\eta (u_k, D_k)\\
  & \qquad  - \eta \limsup_{k \to \infty}\bigg( \int_{\Omega \setminus D_k} |\partial_3 u_k|^p \, \mathrm{d} x + \int_{\Omega \cap \partial^\ast D_k}|\nu_3(D_k)|\, \mathrm{d}\mathcal{H}^2\bigg) \notag. 
\end{flalign}
Noting that $(W^\eta_0)^{\rm qc}(X) \ge (W^0_0)^{\rm qc}(X) = W_0^{\rm qc}(X')$ for every $X \in \R^{3 \times 3}$ we infer from Theorem~\ref{theo:bulk-relaxation}(i) that for every $ \eta > 0 $
\begin{align*}
  \liminf_{k \to \infty} \mathcal{E}_0^\eta (u_k, D_k) \ge \big( \mathcal{E}_0^\eta\big)^{\rm rel}(u,D) 
  \ge \mathcal{E}_0^{\rm rel}(u, D). 
\end{align*}
Since $\bar{c} |(\nabla' u_k, h_k^{-1} \partial_3 u_k)|^p - \bar{C} \le W(\nabla' u_k, h_k^{-1} \partial_3 u_k)$ it follows 
\begin{align*}
  \bar{c} \limsup_{k \to \infty} \int_{\Omega\setminus D_k}|\partial_3 u_k|^p\,\mathrm{d}x 
  & \leq \limsup_{k \to \infty} h_k^{p} \big( \mathcal{E}_k(u_k, D_k) + \bar{C} \mathcal{L}^n(\Omega) \big) 
  = 0. 
\end{align*}
Analogously
\begin{align*}
	\bar{c} \limsup_{k \to \infty} \int_{\Omega\cap \partial^\ast D_k}|\nu_3(D_k)|\,\mathrm{d}\mathcal{H}^2
	& \leq \limsup_{k \to \infty} h_k  \mathcal{E}_k(u_k, D_k) 
	= 0,
\end{align*} 
and the assertion follows from \eqref{eq:lb-est-mem}. 
\end{proof}

\begin{proof}[Proof of Theorem~\ref{theo:membrane-Gamma}(ii)] 
 
We proceed in two steps: Step 1 provides a key auxiliary statement that can be used together with Theorem \ref{theo:bulk-relaxation}(ii) and a standard diagonal sequence argument to construct the recovery sequence.
\smallskip 

\noindent{\em Step 1.} 
Suppose $ (u, D) \in C^\infty(\overline{\omega}; \mathbb{R}^3) \times \mathcal{C}(\omega) $, $c_1, c_2, \ldots \in (0,\mathcal{L}^2(\omega)]$ with $c_k \to \mathcal{L}^2(D)$  and $ \eps > 0 $. In particular, there exists a bounded open set $ A \subseteq \mathbb{R}^2 $ with smooth boundary such that $ D = A \cap \omega $ and $ \mathcal{H}^1(\partial A \cap \partial \omega) =0 $.  We prove that there exists a sequence $(u_k, D_k) \in \C^\infty( \overline{\Omega}; \mathbb{R}^3) \times \mathcal{C}(\Omega)$ such that $ \mathcal{L}^3 (D_k) = c_k $ for each $ k \geq 1 $, $ u_k \to u $ in $ L^1(\Omega; \mathbb{R}^3) $, $ \rchi_{D_k} \to \rchi_{D \times (0,1)} $ in $ L^1(\Omega) $ and
\begin{equation*}
	\limsup_{k \to \infty}\mathcal{E}_k(u_k, D_k) \leq \mathcal{E}_0(u, D) + \eps.
\end{equation*}

Let $ \Gamma_1, \ldots , \Gamma_N $ be the connected components of $ \partial A $ and notice that they are simple smooth curves in $ \mathbb{R}^2 $. For each $ i \in \{1, \ldots , N\} $ we choose a parametrization $ \gamma_i : [0,l] \rightarrow \mathbb{R}^2 $ of $ \Gamma_i $ by arclength and we define $ \eta_i : [0,l_i] \rightarrow \mathbb{R}^2 $ to be the exterior unit-normal of $ A $ along $ \Gamma_i $.  A standard measurable selection criterion \cite[Theorem~1.2 in Chap.~VIII]{EkelandTemam:76} in combination with the lower bound $\psi_0 \geq \bar{c}|\cdot|$  allows us to choose a
 function $ v_i \in L^\infty([0,l_i]) $ such that 
\begin{equation*}
\psi_0(\eta_i(t)) = \psi(\eta_i(t),  v_i(t)) \qquad \textrm{for every $ t \in [0,l_i] $}.
\end{equation*}
Approximating the function $ v_i $ by smooth functions and employing the dominated convergence theorem, we can find
a smooth function $ \phi_i : [0,l_i]\rightarrow \mathbb{R} $ so that the derivatives of $ \phi_i $ satisfy $ \phi_i^{(m)}(0) = \phi_i^{(m)}(l_i) $ for every $ m \geq 0 $ and
\begin{equation*}
\int_{I} \psi(\eta_i(t), \phi_i(t))\, \mathrm{d}t \leq \int_{I}\psi_0(\eta_i(t))\, \mathrm{d}t + \frac{\eps}{2N}
\end{equation*}
for every measurable set $ I \subseteq [0,l_i] $. We consider the ruled surfaces $ \varphi_{i,k} : [0,l_i] \times (0,1) \rightarrow \mathbb{R}^3 $ and $ \varphi_i :[0,l_i] \times (0,1) \rightarrow \mathbb{R}^3 $ of the form
\begin{equation*}
\varphi_{i,k}(t,s) = (\gamma_i(t) - s h_k\phi_i(t)\eta_i(t), s) \quad \textrm{and} \quad \varphi_i(t,s) = (\gamma_i(t), s)
\end{equation*}
and we denote with $ S_{i,k} $ and $ S_i $ their images. One immediately checks that  $ \varphi_{i,k} \to \varphi_i $, $	\partial_t\varphi_{i,k} \to (\gamma_i',0) $ and $\partial_s\varphi_{i,k} \to (0,1) $ uniformly in $ [0,l_i] \times (0,1) $ as $ k \to \infty $; moreover $ D \varphi_{i,k}(t,s) $ is injective for every $ (t,s) \in [0,l_i] \times (0,1) $ and for every $ k \geq 1 $ so that $ h_k \leq \frac{1}{2 \| (\phi_i \eta_i)'\|_\infty} $. Since $ \gamma_i $ is parametrized by arclength, we notice that 
\begin{equation*}
c_i: =	\sup \Big\{ \frac{|\phi_i(t_2)\eta_i(t_2) - \phi_i(t_1)\eta_i(t_1)|}{|\gamma_i(t_2)- \gamma_i(t_1)|} : 0 \leq t_1 < t_2 < l_i \Big\} < \infty
\end{equation*} 
and we conclude that $ \varphi_{i,k}|[0, l_i) \times (0,1) $ is injective for every $ k \geq 1 $ so that $ h_k^{-1} > c_i $. 
Consequently there exists $ k_i \geq 1 $ such that for every $ k \geq k_i $ the set $ S_{i,k} $ is an embedded smooth hypersurface in $ \mathbb{R}^3 $ and the vector
\begin{equation*}
	(\eta_i(t), h_k\phi_i(t)) + \frac{sh_k \phi_i'(t)}{1 - sh_k\phi_i(t) (\eta_i'(t)  \cdot \gamma'_i(t))}(\gamma_i'(t), 0)
\end{equation*}
is normal to $ S_{i,k} $  at $ \varphi_{i,k}(t,s) $. We set $ k_0 = \sup\{k_i: i = 1, \ldots , N\} $ and 
\begin{equation*}
	\lambda_{i,k}(t,s)  = \frac{sh_k \phi_i'(t)}{1 - sh_k\phi_i(t)(\eta_i'(t)  \cdot \gamma'_i(t))}.
\end{equation*}
We can now easily find  a sequence $ \{A_k : k \geq k_0\} $ of open subsets of $ \mathbb{R}^2 \times (0,1) $ such that $ \partial A_k \cap (\mathbb{R}^2 \times (0,1)) = \bigcup_{i=1}^N S_{i,k} $ and $ \rchi_{A_k} \to \rchi_{A \times (0,1)} $ in $ L^1(\mathbb{R}^2 \times (0,1)) $. Noting that
\begin{equation*}
\nu(A_k)(\varphi_{i,k}(t,s)) = \frac{(\eta_i(t), h_k\phi_i(t)) + \lambda_{i,k}(t,s)(\gamma_i'(t),0)}{\sqrt{1 + \lambda_{i,k}(t,s)^2 + h_k^2 \phi_i(t)^2}}
\end{equation*} 
for every $(t,s)\in [0,l_i]\times (0,1) $ and
\begin{flalign*}
&	\int_{\Omega \cap \partial A_k} \psi(\nu'(A_k), h_k^{-1}\nu_3(A_k)) \, \mathrm{d} \mathcal{H}^2 \\
	&~~ =  \sum_{i=1}^N \int_{\varphi_{i,k}^{-1}(\Omega \cap S_{i,k})} \frac{| \partial_t\varphi_{i,k}(t,s) \wedge \partial_s\varphi_{i,k}(t,s)|}{\sqrt{1 + \lambda_{i,k}(t,s)^2 + h_k^2 \phi_i(t)^2}}\psi\big[(\eta_i(t)+ \lambda_{i,k}(t,s)\gamma_i'(t), \phi_i(t))\big] \, \mathrm{d} t \, \mathrm{d} s,
\end{flalign*}
we deduce from the dominated convergence theorem that 
\begin{flalign*}
	& \lim_{k \to \infty}\int_{\Omega \cap \partial A_k} \psi(\nu'(A_k), h_k^{-1}\nu_3(A_k)) \, \mathrm{d} \mathcal{H}^2 \\
	&\quad = \sum_{i=1}^N \int_{\varphi_{i}^{-1}((\Gamma_i \times (0,1)) \cap \Omega)}\psi(\eta_i(t), \phi_i(t)) \, \mathrm{d} t \, \mathrm{d} s \\
	&\quad \leq \sum_{i=1}^N \int_{\varphi_{i}^{-1}(S_i \cap \Omega)} \psi_0(\eta_i(t)) \, \mathrm{d} t \, \mathrm{d} s + \frac{\eps}{2} \\
	&\quad = \int_{\omega \cap \partial A} \psi_0(\nu'(A)) \, \mathrm{d} \mathcal{H}^1 + \frac{\eps}{2}, 
\end{flalign*}
where we have used $\mathcal{H}^1(\partial A \cap \partial \omega) = 0$ in the first equality.

We choose $ x'_0 \in \omega \setminus \overline{D} $ and $ y'_0 \in D $ and we select two nonnegative sequences $ \sigma_k \searrow 0 $ and $ \tau_k \searrow 0 $ and an integer  $ k_0' \geq k_0 $ such that
\begin{equation*}
B_k : = \overline{B_{\sigma_k}(x'_0)} \times (0,1) \subset \Omega \setminus \overline{A_k}, \qquad C_k : =\overline{B_{\tau_k}(y'_0)} \times (0,1) \subset A_k
\end{equation*}
and
\begin{equation*}
	c_k - \mathcal{L}^3(A_k \cap \Omega) = \alpha_2 \sigma_k^2 - \alpha_2 \tau_k^2 = \mathcal{L}^3(B_k) - \mathcal{L}^3(C_k)
\end{equation*}
for all $ k \geq k_0' $. Notice that if $ D = \emptyset $ then $ A_k = \emptyset $ for every $ k \geq 1 $ and we only select $ x'_0 $ and $ \sigma_k $; if $ D = \omega $ then $ A_k \cap \Omega = \omega \times (0,1) $ for every $ k \geq 1 $ and we only select $ y'_0 $ and $ \tau_k $. For each $ k \geq k_0' $ we define
\begin{equation*}
	D_k = \big((A_k \cap \Omega) \cup B_k\big) \setminus C_k \in \mathcal{C}(\Omega)
\end{equation*}
and notice that $ \mathcal{L}^3(D_k) = c_k $, $ \rchi_{D_k} \to \rchi_{D \times (0,1)} $ in $ L^1(\Omega) $ and 
\begin{equation}\label{eq:mem-upper-surf}
\lim_{k \to \infty} \int_{\Omega \cap \partial D_k} \psi(\nu'(D_k), h_k^{-1}\nu_3(D_k)) \, \mathrm{d} \mathcal{H}^2 
\leq \int_{\omega \cap \partial A} \psi_0(\nu'(A)) \, \mathrm{d}\mathcal{H}^1 + \frac{\eps}{2}.
\end{equation}
Following \cite{LeDretRaoult:95} we observe that the measurable selection criterion \cite[Theorem~1.2 in Chap.~VIII]{EkelandTemam:76} in combination with the lower bound in \eqref{eq:p-growth} and the density of $C_c^\infty(\omega \setminus \overline{D}; \mathbb{R}^3)$ in $L^p(\omega \setminus \overline{D}; \mathbb{R}^3) $ allows us to choose $ w \in C_c^\infty(\omega \setminus \overline{D}; \mathbb{R}^3) $ such that 
\begin{equation}\label{eq:mem-upper-bulk}
	\int_{\omega \setminus D} W(\nabla' u, e_3 + w) \, \mathrm{d}x' 
   \leq \int_{\omega \setminus D} W_0(\nabla' u)\, \mathrm{d}x' + \frac{\eps}{2}. 
\end{equation}
We define the functions $ u_k \in C^\infty(\overline{\Omega}; \mathbb{R}^3) $ by
\begin{equation*}
	u_k(x', x_3) = u(x') + x_3 h_k(e_3 + w(x'))
\end{equation*}
for all $(x', x_3) \in \omega \times \mathbb{R} $ and for all  $ k \geq k_0' $ and notice that $ u_k \to u $ in $ L^1(\Omega; \mathbb{R}^3) $. Since $ W $ is continuous, with the help of the dominated convergence theorem we infer 
\begin{equation*}
	\lim_{k \to \infty} \int_{\Omega \setminus D_k}W\big(\nabla'u_k(x'), h_k^{-1} \partial_3 u_k(x')\big) \mathrm{d}x 
    = \int_{\omega \setminus D}W\big(\nabla'u(x'), e_3 + w(x')\big)\, \mathrm{d}x'.  
\end{equation*}
Combining with \eqref{eq:mem-upper-surf} and \eqref{eq:mem-upper-bulk} we get 
\begin{equation*}
    \limsup_{k \to \infty}\mathcal{E}_k\big(u_k, D_k \big) 
    \leq \mathcal{E}_0(u, D) + \eps, 
\end{equation*}
which concludes Step 1.
\smallskip 

\noindent{\em Step 2.} Here we prove Theorem \ref{theo:membrane-Gamma}(ii). Let $(u, D) \in  GSBV^p_1(\omega; \R^{3}) \times \mathcal{F}(\omega)$ and $c_1, c_2, \ldots \in (0,\infty)$ with $c_k \to \mathcal{L}^2(D)$.  Firstly we use Theorem \ref{theo:bulk-relaxation}(ii) to find a sequence $(u_k, D_k) \subset C^\infty(\overline{\omega}; \mathbb{R}^3) \times \mathcal{C}(\omega) $ such that $ \mathcal{L}^2 (D_k) = \mathcal{L}^2(D) $ for each $ k \geq 1 $, $ u_k \to u $ in $ L^1(\omega; \mathbb{R}^3) $, $ \rchi_{D_k} \to \rchi_D $ in $ L^1(\omega) $ and
\begin{equation*}
	\limsup_{k \to \infty}\mathcal{E}_0(u_k, D_k) \leq \mathcal{E}^{\rm rel}_0(u, D).
\end{equation*}
Then for each $ k \geq 1 $ we use Step 1 to find a sequence $(u^{(k)}_h, D^{(k)}_h) \in \C^\infty(\overline{\Omega}; \mathbb{R}^3) \times \mathcal{C}(\Omega)$ such that $ \mathcal{L}^3 (D^{(k)}_h) = c_h $ for each $ h \geq 1 $, $ u^{(k)}_h \to u_k $ in $ L^1(\Omega; \mathbb{R}^3) $, $ \rchi_{D^{(k)}_h} \to \rchi_{D_k \times (0,1)} $ in $ L^1(\Omega) $ and
\begin{equation*}
	\limsup_{h \to \infty}\mathcal{E}_h\big(u^{(k)}_h, D^{(k)}_h \big) 
	\leq \mathcal{E}_0(u_k, D_k) + \frac{1}{k}.
\end{equation*}
At this point a diagonal sequence argument completes the proof. 
\end{proof}


\newcommand{\etalchar}[1]{$^{#1}$}

\end{document}